\documentclass[a4paper,10pt]{article}
\usepackage{LatexDefinitions}
\usepackage{color, booktabs, centernot}
\usepackage{xspace}
\usepackage{graphicx, subcaption}

\usepackage{placeins}

\floatstyle{ruled}
\newfloat{algorithmfloat}{t}{lop}
\floatname{algorithmfloat}{Algorithm}

\newcommand{\proj}{\tn{P}}
\newcommand{\setm}{\mathfrak{C}}
\usepackage{esint}

\renewcommand{\c}{\textup{C}}
\newcommand{\Tor}{\mathbb{T}}

\DeclareMathOperator{\HK}{HK}
\DeclareMathOperator{\Hell}{Hell}
\DeclareMathOperator{\Cos}{Cos}
\DeclareMathOperator{\Sin}{Sin}
\DeclareMathOperator{\dist}{dst}

\newcommand{\weakstar}{\rightharpoonup^*}

\newcommand{\Lebesgue}{\mathcal{L}}

\newcommand{\primalr}{$(\mathcal{P}_{\rho, \kappa})$\xspace}
\newcommand{\dualr}{$(\mathcal{D}_{\rho,\kappa})$\xspace}

\title{Hellinger--Kantorovich barycenter between Dirac measures}
\author{Mauro Bonafini, Olga Minevich, Bernhard Schmitzer}
\date{\today}

\begin{document}
	
	\maketitle
	
	\begin{abstract}
		The Hellinger--Kantorovich (HK) distance is an unbalanced extension of the Wasserstein-2 distance. It was shown recently that the HK barycenter exhibits a much more complex behaviour than the Wasserstein barycenter. Motivated by this observation we study the HK barycenter in more detail for the case where the input measures are an uncountable collection of Dirac measures, in particular the dependency on the length scale parameter of HK, the question whether the HK barycenter is discrete or continuous and the relation between the expected and the empirical barycenter. The analytical results are complemented with numerical experiments that demonstrate that the HK barycenter can provide a coarse-to-fine representation of an input pointcloud or measure.
	\end{abstract}
	
	\section{Introduction}
	\subsection{Motivation}
	\paragraph{Barycenters for Wasserstein and Hellinger--Kantorovich distances}
	Today Wasserstein distances are an ubiquitous tool in mathematical modelling and data analysis, due to their intuitive interpretation, robustness, and rich geometric structure, see \cite{Villani-OptimalTransport-09,Santambrogio-OTAM,PeyreCuturiCompOT} for related monographs on the matter. For instance, the seminal article by Agueh and Carlier \cite{AguehCarlier11} introduces the Wasserstein-2 barycenter as a natural way to average over samples.
	Let $\mu_1,\ldots,\mu_n \in \prob(\Omega)$ be probability measures on some (compact, convex) subset $\Omega \subset \R^d$ and let $\lambda_1,\ldots,\lambda_n$ be positive weights that sum to one. Then the corresponding barycenter problem is to find a minimizer of
	\begin{equation}
		\label{eq:IntroWBarycenter}
		\prob(\Omega) \ni \nu \mapsto \sum_{i=1}^n \lambda_i\,W^2(\mu_i,\nu),
	\end{equation}
	where $W^2$ denotes the squared Wasserstein-2 distance.
	More generally, let $\setm \subset \prob(\Omega)$ be a suitable set of probability measures and let $\Lambda \in \prob(\setm)$ be a probability measure thereon. Then one can generalize problem \eqref{eq:IntroWBarycenter} to
	\begin{equation}
		\label{eq:IntroWBarycenterCont}
		\nu \mapsto \int_{\setm} W^2(\mu,\nu)\,\diff \Lambda(\mu),
	\end{equation}
	i.e.~one might compute the barycenter between a potentially uncountable set of input measures. Such functionals were studied, for example, in \cite{pass2013optimal,BigotKleinBarycenter2018}.
	The discrete case is recovered by setting $\Lambda=\sum_{i=1}^n \lambda_i \cdot \delta_{\mu_i}$.
	
	More recently unbalanced optimal transport distances have received increased attention, motivated (among other things) by better resilience to spurious mass fluctuations. A particularly prominent example in this family is the Hellinger--Kantorovich (HK) distance \cite{KMV-OTFisherRao-2015,ChizatOTFR2015,liero2018optimal}, which exhibits a weak Riemannian structure akin to the Wasserstein-2 distance. The definition of the HK metric involves a length scale parameter $\kappa>0$ that controls the trade-off between mass change and transport. For $\kappa \to \infty$ one recovers the Wasserstein-2 metric, for $\kappa \to 0$ the Hellinger metric.
	The corresponding HK barycenter has been studied in \cite{Chung2021,friesecke2019barycenters}.
	
	\paragraph{HK barycenter between Dirac measures.}
	To obtain better insight into the effect of the length scale parameter, \cite{friesecke2019barycenters} specifically also studied the case where the input measures are all Dirac measures. While this problem would be trivial in the Wasserstein case, a non-trivial structure depending on the length scale parameter $\kappa$ was observed on some simple analytical and preliminary numerical examples.
	The behaviour seemed reminiscent of hierarchical clustering methods where the number of clusters is chosen automatically, depending on $\kappa$.
	Between transitions of different cluster numbers sometimes a diffuse intermediate solution was observed. In these cases, the solution was shown to be non-unique and a discrete solution was always shown to exist as well.
	It is intriguing that a convex functional exhibits such phenomena that are usually associated with non-convex functionals for clustering or measure quantization, see \cite{BourneRoper15} for instance.
	
	Therefore, the goal of this article is to study this barycenter functional and the properties of its minimizers in more detail, with a focus on the case of Dirac marginals.
	Let $\rho \in \measp(\Omega)$ be a probability measure on $\Omega$ that describes the distribution of the input Dirac measures (i.e.~$\rho$ plays the role of $\Lambda$ in \eqref{eq:IntroWBarycenterCont}). Then we study in particular the functional
	\begin{equation}
		\label{eq:IntroHKBarycenterDirac}
		\nu \mapsto \int_\Omega \HK^2_\kappa(\delta_x,\nu)\,\diff \rho(x),
	\end{equation}
	where $\HK^2_\kappa$ denotes the squared HK distance with length scale parameter $\kappa$.
	
	We emphasize again that for the Wasserstein-2 distance this functional is trivially minimized by $\nu=\delta_{\bar{x}}$ with $\bar{x}=\int_\Omega x\,\diff \rho(x)$ being the center of mass of $\rho$ (for simplicity assuming that $\Omega$ is convex). For the HK distance, the functional \eqref{eq:IntroHKBarycenterDirac} exhibits highly non-trivial minimizers, which might be interpretable as clusterings or quantizations of $\rho$ at length scale $\kappa$.
	To this end it will be important to study, whether minimizers of \eqref{eq:IntroHKBarycenterDirac} tend to be sparse; how they evolve under a change of the scale parameter $\kappa$; whether they are robust with respect to variations in $\rho$, e.g.~to show that one obtains a consistent result when an unknown `true' measure $\rho$ is successively approximated by a sequence of empirical measures $\rho_n$ built from sampling; and how to approximately minimize the functional numerically.
	
	\subsection{Outline and contribution}
	Notation and the mathematical setting are fixed in Section \ref{sec:Notation}, and Section \ref{sec:Background} recalls some background on the Wasserstein and Hellinger--Kantorovich distances.
	
	Throughout Section \ref{sec:HKBarycenter} we study the HK barycenter problem between a continuum of general (non-Dirac) input measures. We provide existence and stability under changes in $\Lambda$, see \eqref{eq:IntroWBarycenterCont}, and $\kappa$, including the limits $\kappa \to 0$ or $\infty$. A dual problem is derived that will become instrumental in the analysis of the Dirac case. To our knowledge, this is the first statement of a transport barycenter dual problem for a continuum of input measures.
	
	Section \ref{sec:Diracs} is dedicated to the Dirac case, \eqref{eq:IntroHKBarycenterDirac}. We give simplified expressions for the primal and dual objectives and show existence and uniqueness of dual solutions, primal-dual optimality conditions, and dual stability with respect to $\rho$ and $\kappa$.
	The solution for the $\kappa=0$ limit is given explicitly and the asymptotic behaviour as $\kappa \to 0$ is described in terms of total mass and local mass density of the minimizer.
	Finally, we turn to the question of sparsity of the minimizers. We give an alternative proof to that of \cite{friesecke2019barycenters} that discrete minimizers exist when $\rho$ consists of a finite number of Diracs. But conversely, we give analytical examples for which no discrete minimizers exist.
	
	Section \ref{sec:Numerics} discusses numerical approximation and examples.
	
	We propose a non-convex Lagrangian discretization, reminiscent of methods for quantization problems. It provides high spatial accuracy in the case of sparse solutions. Unlike the quantization problem, missing points can be detected by sampling the dual potential.
	We illustrate that the evolution of the barycenter is stable with respect to $\kappa$, but far from a simple successive merging of clusters. Instead, a wide variety of transition behaviors is documented.
	The convergence of the barycenter as the input data $\rho_n$ converges to $\rho$ is visualized.
	We observe that for some values of $\kappa$ the HK barycenter seems to be unique and can be approximated well numerically, whereas for other values (usually the `transition regimes') this proves to be quite challenging since either it is non-unique or the basin around the minimizer is extremely shallow, as evidenced by very degenerate primal-dual slackness conditions. Non-uniqueness of minimizers and a vast set of near-optimizers are common phenomena in non-convex measure quantization, see e.g.~\cite[Section 4.1]{BourneRoper15}.

	In conclusion, the HK barycenter between Dirac measures does not provide a novel straightforward method for hierarchical point clustering, since the evolution of the minimizer with respect to the length scale parameter does not correspond to a simple successive merging of clusters, and sometimes even only diffuse solutions exist.
	But it does provide an interpolation between the input data and a single Dirac measure, parametrized by a single length scale parameter, that can be interpreted as gradual coarse graining. It is provably stable with respect to the input data and scale changes and comes with a corresponding sequence of dual problems with unique solutions, that provide additional interpretation via the primal-dual optimality relations and information for numerical approximation.
	A summarizing discussion is given in Section \ref{sec:Conclusion}.
	
	\subsection{Notation and setting}
	\label{sec:Notation}
	Throughout this article we adopt the following conventions and notations:
	\begin{itemize}
		\item Let $\Omega \subset \R^d$ be a compact and convex set with non-empty interior.
		\item For a compact metric space $(Y, d_Y)$ we denote by $\c(Y)$ the space of continuous real valued functions equipped with the sup-norm.
		$\meas(Y)$ denotes the space of Radon measures equipped with the total variation norm, the subsets of non-negative and probability measures are denoted by $\measp(Y)$ and $\prob(Y)$ respectively. We consider on $\meas(Y)$ the weak* topology induced via duality with $\c(Y)$. For $\mu \in \measp(Y)$ one has $\|\mu\| = \mu(Y)$.
		\item For a function $f \colon Y \rightarrow\mathbb{R}\cup \{\infty\}$, we denote by $f^*$ its \emph{Fenchel--Legendre conjugate} defined on the dual space $Y^*$ as
		\[
		f^*(y^*) = \sup_{y \in Y} \{ \langle y^*, y \rangle - f(y) \} \quad \text{for } y^* \in Y^*.
		\]	
		\item A measure $\mu\in\meas(Y)$ is \emph{absolutely continuous with respect to} a measure $\nu\in\measp(Y)$, denoted $\mu \ll \nu$, if for every measurable subset $A\subset Y$, $\nu(A) = 0$ implies $\mu(A) = 0$. For $\mu \ll \nu$, we denote by $\diff\mu/\diff\nu$ the \emph{Radon--Nikodym derivative} of $\mu$ w.r.t.~$\nu$.
		\item Given two compact metric spaces $X,Y$, a measurable function $f \colon X \to Y$ and a measure $\mu \in \meas(X)$, we define the \emph{push-forward} of $\mu$ through $f$ as the measure $f_\#\mu \in \meas(Y)$ characterized by
		\[
		\int_Y \phi(y)\,\diff (f_\#\mu)(y) = \int_X \phi(f(x))\,\diff \mu(x) \quad \tn{for all }\phi \in \c(Y).
		\]
		
		\item For $\mu \in \meas(Y)$, $\nu \in \measp(Y)$ the \emph{Kullback--Leibler divergence} (or relative entropy) of $\mu$ w.r.t.~$\nu$ is given by
		\begin{align*}
			\KL(\mu|\nu) & \assign \begin{cases}
				\displaystyle \int_Y \varphi\left(\RadNik{\mu}{\nu}\right)\,\diff \nu & \tn{if } \mu \ll \nu,\,\mu \geq 0, \\
				+ \infty & \tn{else,}
			\end{cases} \\
			\intertext{where $\varphi \colon \R \to \RCupInf$ is defined by}
			\varphi(s) & \assign \begin{cases}
				s\,\log(s)-s+1 & \tn{if } s>0, \\
				1 & \tn{if } s=0, \\
				+ \infty & \tn{else.}
			\end{cases} \\
		\end{align*}
		
	\end{itemize}
	
	\section{Wasserstein and Hellinger--Kantorovich metrics}
	\label{sec:Background}
	\subsection{Wasserstein distance}
	Let $(Y,d_Y)$ be a compact metric space. For $\mu, \nu \in \measp(Y)$, we define the \emph{Wasserstein distance} $W$ between $\mu$ and $\nu$ as
	\begin{align}
		W^2(\mu,\nu) & \assign \inf \left\{ \int_{Y \times Y} d_Y^2(x,y)\,\diff \gamma(x,y) \middle| \gamma \in \Gamma(\mu,\nu) \right\},
	\end{align}
	where
	\begin{align}
		\Gamma(\mu,\nu) & \assign \left\{ \gamma \in \measp(Y \times Y) \middle| \proj_1 \gamma=\mu,\,\proj_2\gamma=\nu \right\}.
	\end{align}
	Here $\proj_1$ and $\proj_2$ denote the operators that map measures on $\meas(Y \times Y)$ to their first and second marginal respectively, i.e. $P_1\gamma = [(x,y)\mapsto x]_\#\gamma$ and $P_2\gamma = [(x,y)\mapsto y]_\#\gamma$.
	Note that $\Gamma(\mu,\nu) \neq \emptyset$ if and only if $\mu(Y)=\nu(Y)$. Therefore, $W(\mu,\nu)$ is finite if and only if $\|\mu\|=\|\nu\|$ and by convention $W(\mu,\nu)=+\infty$ otherwise. When restricted to $\prob(\Omega)$, $W$ gives the well-known Wasserstein-2 distance.
	
	A dual formulation is given by
	\begin{multline}
		W^2(\mu, \nu) = \sup \left\{ \int_Y \psi(x)\,\diff \mu(x) + \int_Y \phi(y)\,\diff\nu(y) \right| \\
		\left. \vphantom{\int_{Y}} \psi, \phi \in \c(Y), \psi(x) + \phi(y) \leq d_Y^2(x,y) \text{ for all } x,y \in Y \right\}.
	\end{multline}
	Again, note that for $\|\mu\| \neq \|\nu\|$ the supremum is $+\infty$, consistent with our convention for the primal formulation. We recall here the basic properties of this distance.
	
	\begin{theorem}[Basic properties of $W$ \protect{\cite[Theorems 6.9, 6.18]{Villani-OptimalTransport-09}}]
		\label{thm:W2Basic}
		Let $(Y,d_Y)$ be a compact metric space. The Wasserstein distance $W$ metrizes the weak* topology over $\prob(Y)$. The metric space $(\prob(Y), W)$ is separable and complete.
	\end{theorem}
	
	\subsection{Hellinger--Kantorovich distance}
	While Wasserstein distances only allow a meaningful comparison between measures of equal mass, the Hellinger--Kantorovich distance is a (geodesic) metric on the set of all non-negative measures. We briefly recall the properties required in this article. An in-depth study is provided in \cite{liero2018optimal}, a compact summary of several of its equivalent formulations, its geodesics and a comparison with the Wasserstein distance can be found in \cite{LinHK2021}.
	
	For $\mu, \nu \in \measp(\Omega)$ the scaled \emph{Hellinger--Kantorovich distance} $\HK_\kappa$ is given by \cite{liero2018optimal}
	\begin{align}
		\label{eq:HKSoftMarginal}
		\HK_\kappa^2(\mu,\nu) & \assign
		\inf \left\{ \int_{\Omega^2} \hat{c}_\kappa(x,y)\,\diff \gamma(x,y) + \KL(\proj_1 \gamma|\mu) + \KL(\proj_2 \gamma|\nu)
		\middle| \gamma \in \measp(\Omega^2) \right\}, \\
		\intertext{where}
		\hat{c}_\kappa(x,y) & \assign \begin{cases} -2 \log \cos(|x-y|/\kappa) & \tn{for } |x-y| \leq \kappa\pi/2 ,\\
			+ \infty & \tn{otherwise.} \end{cases}
	\end{align}
	This is an optimal transport problem where the marginal constraints are relaxed and deviations from the marginals $\mu$ and $\nu$ are admissible and penalized by the Kullback--Leibler divergence, allowing for changes of mass. The parameter $\kappa>0$ is a length scale parameter that balances the trade-off between transport and mass change. From the definition of $\hat{c}_\kappa$ we infer that mass is never transported further than $\kappa\pi/2$, in particular $\kappa$ effectively re-scales the Euclidean distance on $\Omega$, as elaborated in the following Remark.
	\begin{remark}
		\label{rem:HKKappaRescaling}
		For $\kappa \in (0,\infty)$ let $S: \Omega \to \Omega/\kappa$, $x \mapsto x/\kappa$. Then for $\mu, \nu \in \measp(\Omega)$ one obtains that $\HK_\kappa^2(\mu,\nu)=\HK_1^2(S_\# \mu,S_\# \nu)$ where the latter distance is computed on $\measp(\Omega/\kappa)$. This follows quickly from the fact that $S$ is a homeomorphism between $\Omega$ and $\Omega/\kappa$, implying for instance $\KL(\proj_1 \gamma|\mu)=\KL(\proj_1 S_\#\gamma|S_\# \mu)$.
	\end{remark}
	
	We also recall an alternative formulation for $\HK_\kappa$, as given by \cite{ChizatDynamicStatic2018}.
	Let $\Cos \colon \R \to \R$ denote the truncated cosine function defined as
	\begin{align*}
		\Cos(s) \assign \cos\left(\min\{|s|, \; \pi/2\}\right) \quad \tn{for } s \in \R
	\end{align*}
	and consider the following cost function $c_\kappa \colon \Omega \times \R \times \Omega \times \R \to \RCupInf$,
	\[
	c_\kappa(x_1, m_1, x_2, m_2) \assign
	\begin{cases}
		m_1 + m_2 - 2\sqrt{m_1m_2} \Cos(|x_1-x_2|/\kappa) & \tn{if } m_1, m_2 \geq 0, \\
		+ \infty & \tn{otherwise.} \end{cases} \\
	\]
	Then, the scaled Hellinger--Kantorovich distance $\HK_\kappa$ can be written as
	\begin{multline}
		\label{eq:HKSemiCoupling}
		\HK_\kappa^2(\mu,\nu) =
		\inf \Bigg\{ \int_{\Omega^2} c_\kappa\left(x,\tfrac{\diff \gamma_1}{\diff \gamma}(x,y),y,\tfrac{\diff \gamma_2}{\diff \gamma}(x,y)\right)\,\diff \gamma(x,y)	\Bigg| 
		\\ \gamma_1, \gamma_2, \gamma \in \measp(\Omega^2),\, \gamma_1, \gamma_2 \ll \gamma
		\tn{ and } \proj_1\gamma_1 = \mu, \proj_2\gamma_2=\nu\Bigg\}.
	\end{multline}
	Note that $\gamma$ in \eqref{eq:HKSemiCoupling} is just an auxiliary variable and the integral does not depend on the choice of $\gamma$ by positive 1-homogeneity of $c_\kappa$ in its second and fourth argument.
	
	A dual formulation for \eqref{eq:HKSoftMarginal} and \eqref{eq:HKSemiCoupling} is given by \cite{ChizatDynamicStatic2018}
	\begin{align}
		\label{eq:HKDual}
		\HK_\kappa^2(\mu, \nu) & = \sup_{(\psi, \phi) \in Q_\kappa} \left[ \int_\Omega \psi(x)\,\diff \mu(x) + \int_\Omega \phi(y)\,\diff\nu(y) \right],
	\end{align}
	where the set $Q_\kappa$ is defined by
	\begin{align}
		Q_\kappa \assign \left\{(\psi,\phi) \in \c(\Omega) \times \c(\Omega) \,\text{ s.t. }\,
		\begin{aligned}
			&\psi(x), \phi(y) \in (-\infty,1], \\ &(1-\psi(x))(1-\phi(y))\geq \Cos^2(|x-y|/\kappa)
		\end{aligned}
		\quad \text{for all } x,y \in \Omega\right\}
		\label{eq:Qkappa}.
	\end{align}

	\begin{theorem}[Basic properties of $\HK_\kappa$]
		\label{thm:HKBasic}
		For any $\kappa \in (0,\infty)$, the Hellinger--Kantorovich distance $\HK_\kappa$ metrizes the weak* topology over $\measp(\Omega)$.
		The metric space $(\measp(\Omega), \HK_k)$ is separable and complete. Furthermore, it is a proper metric space, i.e.~every bounded set is relatively compact (see \cite[Section~7.5]{liero2018optimal}).
	\end{theorem}
	
	As we seek to study the evolution of the Hellinger--Kantorovich barycenter over varying length scales we now recall the corresponding result.
	For $\mu,\nu \in \measp(\Omega)$ the \emph{Hellinger distance} is defined as
	\begin{align}
		\label{eq:Hell}
		\Hell^2(\mu, \nu) \assign \int_\Omega \left( \sqrt{\frac{\diff \mu}{\diff \tau}} - \sqrt{\frac{\diff \nu}{\diff \tau}} \right)^2 \!\diff\tau,
	\end{align}
	where $\tau \in \measp(\Omega)$ is an arbitrary measure such that $\mu,\nu \ll \tau$. Again, since the function $(s,t) \mapsto (\sqrt{s}-\sqrt{t})^2$ is positively 1-homogeneous, the definition of $\Hell(\mu,\nu)$ does not depend on the choice of the (admissible)~$\tau$. As observed in \cite{liero2018optimal}, the Hellinger--Kantorovich distance converges towards the Hellinger and the Wasserstein distance as one sends $\kappa \to 0$ or $\kappa \to \infty$ respectively.
	
	\begin{theorem}[Scaling limits \protect{\cite[Theorems 7.22, 7.24]{liero2018optimal}}]
		\label{thm:HKScaling}
		For $\mu, \nu \in \measp(\Omega)$, one finds that the function $(0,\infty) \ni \kappa \mapsto \HK_\kappa^2(\mu,\nu)$ is non-increasing and
		\begin{align}
			\label{eq:Helllimit}
			\lim_{\kappa \to 0} \HK^2_\kappa(\mu,\nu) \nearrow \Hell^2(\mu,\nu).
		\end{align}
		On the other hand, the function $(0,\infty) \ni \kappa \mapsto \kappa^2 \cdot \HK_\kappa^2(\mu,\nu)$ is non-decreasing and
		\begin{align}
			\label{eq:W2limit}
			\lim_{\kappa \to \infty} \kappa^2 \HK^2_\kappa(\mu,\nu) \nearrow W^2(\mu,\nu).
		\end{align}
		The function $(0,\infty) \ni \kappa \mapsto \HK^2_\kappa(\mu,\nu)$ is continuous.
	\end{theorem}
	Continuity of $\HK^2_\kappa$ with respect to $\kappa$ follows directly from the fact that the function is non-increasing while $\kappa \mapsto \kappa^2 \cdot \HK^2_\kappa(\mu,\nu)$ is non-decreasing.
	Note that the case $\|\mu\|\neq \|\nu\|$ is explicitly allowed as $\kappa \to \infty$, in which case the limiting value is $+\infty$.
	These scaling limits can be guessed from \eqref{eq:HKSoftMarginal}: As $\kappa \to 0$, the function $\hat{c}_\kappa$ goes to infinity everywhere except on the diagonal, restricting asymptotically feasible $\gamma$ to the diagonal. One can then quickly verify that minimizing \eqref{eq:HKSoftMarginal} only over diagonal $\gamma$ yields \eqref{eq:Hell}.
	Conversely, looking at $\kappa^2 \cdot \HK_\kappa^2(\mu,\nu)$, one can guess from $\lim_{\kappa \to \infty} \kappa^2 \cdot \hat{c}_\kappa(x,y) = \|x-y\|^2$ that the integral $\int_{\Omega^2} \hat{c}_\kappa \diff \gamma$ converges to the standard Wasserstein transport cost, while the term $\kappa^2 \cdot \KL(\proj_1 \gamma|\mu)$ increasingly penalizes deviations between $\proj_1 \gamma$ and $\mu$, and likewise for the second marginal term, thus asymptotically enforcing $\gamma \in \Gamma(\mu,\nu)$.
	
	The following bounds can be shown to hold. 
	\begin{proposition}[Mass-rescaling for $\HK_\kappa$]
		For $\kappa \in (0,\infty)$ and $\mu, \nu \in \measp(\Omega)$, one finds
		\begin{equation}
			\label{eq:rescale}
			\HK_\kappa^2(\mu, \nu) = \sqrt{\|\mu\|\|\nu\|}\HK_\kappa^2\left( \frac{\mu}{\|\mu\|}, \frac{\nu}{\|\nu\|} \right) + (\sqrt{\|\mu\|} - \sqrt{\|\nu\|})^2
		\end{equation}
		with the convention $\mu/\|\mu\|=0$ in the case of $\mu=0$ (and likewise for $\nu$).
		Additionally,
		\begin{equation}
			\label{eq:bounds}
			\HK_\kappa^2(\mu, \nu) \leq {\|\mu\|} + {\|\nu\|}.
		\end{equation}
	\end{proposition}
	
	\begin{proof}
		The equality in \eqref{eq:rescale} follows from \cite[Theorem 3.3]{Laschos2019}. By Theorem \ref{thm:HKScaling}, we have $\HK_\kappa^2\left( \mu, \nu \right)\leq \Hell^2( \mu, \nu )$, and \eqref{eq:bounds} follows directly.
	\end{proof}
	
	\section{Hellinger--Kantorovich barycenter of a continuum of measures}
	\label{sec:HKBarycenter}
	\subsection{Problem setup}
	The barycenter between a finite collection of measures with respect to the Hellinger--Kantorovich metric has been studied in \cite{Chung2021,friesecke2019barycenters}. In this section we generalize these results to infinitely many input measures, including the uncountable case of a continuum of input measures.
	
	For a constant $M \in (0,\infty)$, which we will assume to be fixed throughout the paper, we define
	\[
	\setm \assign \{ \mu \in \measp(\Omega) | \|\mu\| \leq M \}.
	\]
	Since $\setm$ is weak* closed, by Theorem \ref{thm:HKBasic} the metric space $(\setm,\HK_\kappa)$ is compact for all $\kappa \in (0,\infty)$. We will describe the collection of input measures (and their weights) of which to compute the barycenter as a probability measure $\Lambda \in \prob(\setm)$ where $\setm$ is equipped with the Borel $\sigma$-algebra induced by $\HK_\kappa$ (which is the same for any $\kappa \in (0,\infty)$).
	Since $(\setm,\HK_\kappa)$ is compact, weak* convergence on $\prob(\setm)$ is metrized by the Wasserstein distance over $\prob(\setm)$ (see Theorem \ref{thm:W2Basic}).
	
	\medskip
	
	For $\Lambda \in \prob(\setm)$ and for $\kappa \in (0,\infty)$, the primal problem we are interested in is
	\begin{equation}\label{eq:primal}
		\inf \left\{ J_{\Lambda, \kappa}(\nu) \assign \int_{\setm} \HK_{\kappa}^2(\mu,\nu)\,\diff\Lambda(\mu) \middle| \nu \in \measp(\Omega) \right\}.
		\tag{$\mathcal{P}_{\Lambda,\kappa}$}
	\end{equation}
	The finite case of computing the barycenter between input measures $\mu_1,\ldots,\mu_n \in \setm$ with weights $\lambda_1,\ldots,\lambda_n$ where $\lambda_i > 0$ and $\sum_{i=1}^n \lambda_i=1$ is recovered by setting $\Lambda \assign \sum_{i=1}^n \lambda_i\,\delta_{\mu_i}$.
	
	\subsection{Existence and stability of minimizers}
	\label{sec:ExistenceStability}
	\begin{proposition}\label{prop:primalexistence}
		Let $\Lambda \in \prob(\setm)$ and $\kappa \in (0, \infty)$. Then, \eqref{eq:primal} admits a minimizer $\nu \in \setm$.
	\end{proposition}
	
	\begin{proof}
		We first observe that, by means of the upper bound in \eqref{eq:bounds}, we have
		\[
		\text{\eqref{eq:primal} } \leq J_{\Lambda,\kappa}(0) = \int_{\setm} \HK_\kappa^2(\mu,0)\,\diff\Lambda(\mu) \leq \int_{\setm} \|\mu\| \,\diff\Lambda(\mu) \leq M.
		\]
		Let $(\nu_n)_n \subset \measp(\Omega)$ be a minimizing sequence for \eqref{eq:primal}. For each $n > 0$, we can assume without loss of generality that $\nu_n \in \setm$. Indeed, suppose this is not the case, i.e. $\|\nu_n\| > M$. Then, for all $\mu \in \setm$, by means of \eqref{eq:rescale}, we have
		\begin{align*}
			\HK_\kappa^2\left( \mu, \frac{M}{\|\nu_n\|}\nu_n \right)
			&=
			\sqrt{\|\mu\|M}\,\HK_\kappa^2\left( \frac{\mu}{\|\mu\|}, \frac{\nu_n}{\|\nu_n\|} \right) + (\sqrt{\|\mu\|} - \sqrt{M})^2
			\\
			&<
			\sqrt{\|\mu\|\|\nu_n\|} \, \HK_\kappa^2\left( \frac{\mu}{\|\mu\|}, \frac{\nu_n}{\|\nu_n\|} \right) + (\sqrt{\|\mu\|} - \sqrt{\|\nu_n\|})^2 = \HK_\kappa^2(\mu, \nu_n),
		\end{align*}
		so that $J_{\Lambda, \kappa}(M/\|\nu_n\| \cdot \nu_n) < J_{\Lambda, \kappa}(\nu_n)$. Hence, upon possibly replacing $\nu_n$ with $M/\|\nu_n\| \cdot \nu_n$, the sequence $(\nu_n)_n$ is entirely contained in $\setm$. By compactness of $(\setm, \HK_\kappa)$, there exists a cluster point $\nu \in \setm$ such that, up to a subsequence, $\nu_n \weakstar \nu$ as $n \to \infty$, or equivalently $\HK_\kappa(\nu_n,\nu) \to 0$ as $n \to \infty$. By the upper bound in \eqref{eq:bounds} and by triangle inequality we also have
		\begin{align}
			|\HK_\kappa^2(\mu,\nu_n) - \HK_\kappa^2(\mu,\nu)| &= |\HK_\kappa(\mu,\nu_n) + \HK_\kappa(\mu,\nu)||\HK_\kappa(\mu,\nu_n) - \HK_\kappa(\mu,\nu)| \nonumber \\
			&\leq 2\sqrt{2M} \cdot \HK_\kappa(\nu_n,\nu) \to 0 \quad \text{as } n \to \infty \text{ for all } \mu \in \setm.
			\label{eq:HKUniformConv}
		\end{align}
		By means of Fatou's lemma and recalling that $(\nu_n)_n$ is a minimizing sequence, we conclude
		\[
		J_{\Lambda,\kappa}(\nu) = \int_{\setm} \HK_\kappa^2(\mu,\nu)\,\diff\Lambda(\mu)
		= \int_{\setm} \lim_{n \to \infty} \HK_\kappa^2(\mu,\nu_n)\,\diff\Lambda(\mu)
		\leq \liminf_{n \to \infty} \underbrace{\int_{\setm} \HK_\kappa^2(\mu,\nu_n)\,\diff\Lambda(\mu)}_{J_{\Lambda,\kappa}(\nu_n)} = \text{\eqref{eq:primal}},
		\]
		which provides minimality of $\nu$ for \eqref{eq:primal}.
	\end{proof}

	\begin{proposition}[Stability]
		\label{prop:stability}
		Fix $\kappa \in (0, \infty)$. Let $(\Lambda_n)_{n \in \N}$ be a weak* convergent sequence in $\prob(\setm)$ with limit $\Lambda \in \prob(\setm)$ and let $(\nu_n)_{n \in \N}$ be a weak* convergent sequence in $\measp(\Omega)$ with limit $\nu \in \measp(\Omega)$. Then, 
		\begin{equation}\label{eq:limitenergy}
			J_{\Lambda,\kappa}(\nu) = \lim_{n \to \infty} J_{\Lambda_n,\kappa}(\nu_n).
		\end{equation}
	\end{proposition}
	
	\begin{proof}
		As in \eqref{eq:HKUniformConv}, the sequence of functions $(\mu \mapsto \HK^2_\kappa(\mu,\nu_n))_n$ converges uniformly to $(\mu \mapsto \HK^2_\kappa(\mu,\nu))$ in $\c(\setm)$. This, together with weak* convergence of $\Lambda_n$ to $\Lambda$, leveraging duality between $\c(\setm)$ and $\measp(\setm)$, leads to
		\begin{align*}
			\lim_{n \to \infty} J_{\Lambda_n,\kappa}(\nu_n)
			=
			\lim_{n \to \infty} \int_{\setm} \HK_\kappa^2(\mu,\nu_n)\,\diff\Lambda_n(\mu)
			= \int_{\setm} \HK_\kappa^2(\mu,\nu)\,\diff\Lambda(\mu)
			= J_{\Lambda,\kappa}(\nu),
		\end{align*}
		which provides \eqref{eq:limitenergy}.
	\end{proof}
	
	\begin{corollary}[Convergence of minimizers]
		\label{cor:minimizersconvergence}
		Fix $\kappa \in (0, \infty)$. Let $(\Lambda_n)_{n \in \N}$ be a weak* convergent sequence in $\prob(\setm)$ with limit $\Lambda \in \prob(\setm)$ and, for each $n$, let $\nu_n$ be a minimizer of $(\mathcal{P}_{\Lambda_n, \kappa})$. Then, the sequence $(\nu_n)_{n \in \N}$ is weak* pre-compact and each cluster point $\nu \in \setm$ is a minimizer of \eqref{eq:primal}.
	\end{corollary}
	
	\begin{proof}
		By Proposition \ref{prop:primalexistence}, the sequence of minimizers $(\nu_n)_{n \in \N}$ lies entirely in $\setm$, hence by compactness of $(\setm, \HK_\kappa)$ it is weak* pre-compact.
		Fix now any weak* cluster point $\nu$ of $(\nu_n)_n$ and a corresponding subsequence $(\nu_{n'})_{n'}$ such that $\nu_{n'} \weakstar \nu$ as $n' \to \infty$. Fix any $\tilde{\nu} \in \measp(\Omega)$. Minimality of each $\nu_n$ for $(\mathcal{P}_{\Lambda_n, \kappa})$ and a double application of Proposition \ref{prop:stability} provide
		\[
		J_{\Lambda,\kappa}(\nu) = \lim_{n' \to \infty} J_{\Lambda_{n'},\kappa}(\nu_{n'}) \leq \lim_{n' \to \infty} J_{\Lambda_{n'},\kappa}(\tilde{\nu}) = J_{\Lambda,\kappa}(\tilde{\nu}),
		\]
		which proves minimality of $\nu$ for \eqref{eq:primal}.
	\end{proof}
	
	\subsection{Scaling limits for the metric}
	\label{sec:Scaling}
	Now let us look at the limit problems as we send $\kappa \to 0$ and $\kappa \to \infty$ respectively.
	Based on Theorem~\ref{thm:HKScaling} we expect to recover the pure Hellinger and pure Wasserstein barycenter problems (after suitable re-scaling). The expected limit functionals are therefore:
	\begin{align}
		J_{\Lambda,0}(\nu) &
		\assign \int_\setm \textup{Hell}^2(\mu, \nu)\,\diff \Lambda(\mu), &
		J_{\Lambda,\infty}(\nu) &
		\assign \int_\setm W^2(\mu, \nu)\,\diff \Lambda(\mu).
	\end{align}
	In particular, we obtain as a by-product the existence of minimizers for such limiting barycenter problems.
	
	\begin{proposition}\label{prop:scalingLimits}
		Let $\Lambda \in \prob(\setm)$, let $(\kappa_n)_n$ be a sequence in $(0,\infty)$ with $\lim_n \kappa_n = \kappa_\infty \in [0,\infty) \cup \{\infty\}$. For each $n$, let $\nu_n \in \setm$ be a minimizer of $J_{\Lambda,\kappa_n}$. Then, the sequence $(\nu_n)_n$ is weak* pre-compact and each cluster point $\nu_\infty \in \setm$ is a minimizer of $J_{\Lambda,\kappa_\infty}$. Furthermore,
		\begin{align}
			\label{eq:energylimitfinite}
			J_{\Lambda,\kappa_\infty}(\nu_\infty) & = \lim_{n \to \infty} J_{\Lambda,\kappa_n}(\nu_n) \quad \text{if} \quad \kappa_\infty \in [0,\infty)
			\intertext{and}
			\label{eq:energylimitinfinity}
			J_{\Lambda,\infty}(\nu_\infty) & = \lim_{n \to \infty} \kappa_n^2 J_{\Lambda,\kappa_n}(\nu_n) \quad \text{if} \quad \kappa_\infty = \infty.
		\end{align}
	\end{proposition}
	
	\begin{proof}
		By Proposition \ref{prop:primalexistence} the sequence $(\nu_n)_n$ lies in $\setm$, thus has uniformly bounded mass and thus is weak* pre-compact.
		Let us assume for now that the sequence $(\kappa_n)_n$ is monotone and that the corresponding sequence $(\nu_n)_n$ converges weak* to $\nu_\infty \in \setm$.

		\medskip
		\noindent \emph{Step 1.1 ($\kappa_n \nearrow \kappa_\infty$)}
		\smallskip
		
		\noindent
		Assume the sequence $(\kappa_n)_n$ is non-decreasing and converging to $\kappa_\infty \in (0,\infty)$, and assume the corresponding sequence $(\nu_n)_n$ converges weak* to $\nu_\infty \in \setm$. By Theorem \ref{thm:HKScaling} the function $(0,\infty) \ni \kappa \mapsto \kappa^2 \cdot \HK_\kappa^2(\mu,\nu)$ is non-decreasing for all $\mu,\nu \in \setm(\Omega)$ and $\lim_{n \to \infty} \kappa_n^2\,\HK_{\kappa_n}^2(\mu,\nu)\allowbreak \nearrow \kappa_\infty^2\,\HK_{\kappa_\infty}^2(\mu,\nu)$. Therefore, for each $\nu \in \setm$, the function $(0,\infty) \ni \kappa \mapsto \kappa^2\,J_{\Lambda,\kappa}(\nu)$ is non-decreasing, so that
		\begin{equation}\label{eq:monotonicityJ}
			\kappa_m^2\,J_{\Lambda,\kappa_m}(\nu) \leq \kappa_n^2\,J_{\Lambda,\kappa_n}(\nu) \quad \text{for all }n > m > 0, \text{ and all } \nu \in \measp(\Omega),
		\end{equation}
		and by monotone convergence
		\begin{equation}\label{eq:energyConvergenceFixednu}
			\kappa_n^2\,J_{\Lambda,\kappa_n}(\nu) \nearrow \kappa_\infty^2 J_{\Lambda,\infty}(\nu) \text{ as } n \to \infty, \text{ for all } \nu \in \measp(\Omega).
		\end{equation}
		Let us fix $m \in \N$. Thanks to \eqref{eq:monotonicityJ}, applied with $\nu = \nu_n$, we find
		\[
		\kappa_m^2\,J_{\Lambda,\kappa_m}(\nu_n) \leq \kappa_n^2\,J_{\Lambda,\kappa_n}(\nu_n) \quad \text{for all $n > m$},
		\]
		so that, passing to the limit as $n \to \infty$, we obtain
		\[
		\kappa_m^2\,J_{\Lambda,\kappa_m}(\nu_\infty) \overset{\eqref{eq:limitenergy}}{=} \kappa_m^2 \lim_{n \to \infty} J_{\Lambda,\kappa_m}(\nu_n) \leq \liminf_{n \to \infty} \kappa_n^2\,J_{\Lambda,\kappa_n}(\nu_n).
		\]
		Passing now to the limit as $m \to \infty$ we conclude
		\begin{equation}\label{eq:limsup}
			\kappa_\infty^2 J_{\Lambda,\kappa_\infty}(\nu_\infty) \overset{\eqref{eq:energyConvergenceFixednu}}{=} \lim_{m \to \infty} \kappa_m^2\,J_{\Lambda,\kappa_m}(\nu_\infty) \leq \liminf_{n \to \infty} \kappa_n^2\,J_{\Lambda,\kappa_n}(\nu_n).
		\end{equation}
		On the other hand, using \eqref{eq:energyConvergenceFixednu} with $\nu = \nu_\infty$ and leveraging minimality of each $\nu_n$, one has
		\[
		\kappa_\infty^2 J_{\Lambda,\kappa_\infty}(\nu_\infty) \geq \kappa_n^2\,J_{\Lambda,\kappa_n}(\nu_\infty) \geq \kappa_n^2\,J_{\Lambda,\kappa_n}(\nu_n) \quad \text{for every }n,
		\]
		and a passage to the limit as $n \to \infty$ directly provides
		\begin{equation}\label{eq:liminf}
			\kappa_\infty^2 J_{\Lambda,\kappa_\infty}(\nu_\infty) \geq \limsup_{n \to \infty} \kappa_n^2\,J_{\Lambda,\kappa_n}(\nu_n).
		\end{equation}
		Combining \eqref{eq:liminf} and \eqref{eq:limsup} provides \eqref{eq:energylimitfinite} for the particular class of sequences under consideration. Assume now $\nu_\infty$ were not optimal for $J_{\Lambda,\kappa_\infty}$, i.e. there exists some $\nu_\infty'$ with a strictly better score. By minimality of each $\nu_n$, we have $J_{\Lambda,\kappa_n}(\nu_n) \leq J_{\Lambda,\kappa_n}(\nu_\infty')$ so that, passing to the limit one has
		\[
		\lim_{n \to \infty} J_{\Lambda,\kappa_n}(\nu_n) \leq \lim_{n \to \infty} J_{\Lambda,\kappa_n}(\nu_\infty') \overset{\eqref{eq:energyConvergenceFixednu}}{=} J_{\Lambda,\kappa_\infty}(\nu_\infty') < J_{\Lambda,\kappa_\infty}(\nu_\infty) \overset{\eqref{eq:energylimitfinite}}{=} \lim_{n \to \infty} J_{\Lambda,\kappa_n}(\nu_n),
		\]
		hence the sought for contradiction. The same argument applies if $\kappa_n \nearrow \infty$, taking into account that $\lim_{n \to \infty} \kappa_n^2\,\HK_{\kappa_n}^2(\mu,\nu)\allowbreak \nearrow W^2(\mu,\nu)$ for all $\mu,\nu \in \setm(\Omega)$.
		
		\medskip\noindent
		\emph{Step 1.2 ($\kappa_n \searrow \kappa_\infty$)}
		
		\smallskip
		\noindent
		If we assume instead that the sequence $(\kappa_n)_n$ is non-increasing and $\kappa_n \searrow \kappa_\infty \in [0,\infty)$, a completely symmetric argument as in Step 1.1 can be applied after dropping the scaling factors $\kappa_n^2$. Indeed, by Theorem~\ref{thm:HKScaling} the function $(0,\infty) \ni \kappa \mapsto \HK_\kappa^2(\mu,\nu)$ is non-increasing for all $\mu,\nu \in \setm(\Omega)$ and so $\lim_{n \to \infty} \HK_{\kappa_n}^2(\mu,\nu)\allowbreak \nearrow \HK_{\kappa_\infty}^2(\mu,\nu)$ (with limit $\Hell^2(\mu,\nu)$ if $\kappa_\infty = 0$). Thus, the same monotonicity arguments apply.
		
		\medskip\noindent
		\emph{Step 2}.
		Assume now $\nu_\infty \in \setm$ is any cluster point of $(\nu_n)_n$. Hence, there exists a subsequence $(\nu_{n'})_{n'}$ such that $\nu_{n'} \weakstar \nu_\infty$ as $n' \to \infty$. We can extract an additional subsequence such that $(\kappa_{n''})_{n''}$ is either non-increasing or non-decreasing. Step 1 then provides minimality of $\nu_\infty$ for $(\mathcal{P}_{\Lambda, \kappa_\infty})$.
		
		We are left to prove that the sequence of energies $(J_{\Lambda,\kappa_{n}}(\nu_{n}))_{n}$ converges as a whole. Consider any subsequence $(J_{\Lambda,\kappa_{n'}}(\nu_{n'}))_{n'}$. By pre-compactness of the corresponding sequence $(\nu_{n'})_{n'}$, we can identify a further subsequence such that $(\nu_{n''})_{n''}$ converges weak* to some cluster point $\nu_\infty \in \setm$. and in turn extract an additional subsequence such that $(\kappa_{n'''})_{n'''}$ is either non-increasing or non-decreasing. By Step~1 we have
		\[
		\lim_{n''' \to \infty} J_{\Lambda,n'''}(\nu_{n'''}) = (\mathcal{P}_{\Lambda, \kappa_\infty}).
		\]
		Hence, every subsequence of $(J_{\Lambda,\kappa_n}(\nu_n))_n$ admits a converging subsequence to the same limit $(\mathcal{P}_{\Lambda, \kappa_\infty})$. This provides convergence of the full sequence to $(\mathcal{P}_{\Lambda, \kappa_\infty})$ and proves \eqref{eq:energylimitfinite} for the whole sequence of minimizing energies.
	\end{proof}
	
	For $\kappa \in (0,\infty]$ it is known that barycenters are not necessarily unique (see, e.g., \cite[Section 6]{friesecke2019barycenters}), hence there may be multiple corresponding cluster points $\nu_\infty$ in the above result. We now show that for $\kappa=0$ uniqueness holds in general.
	
	\begin{corollary}[Convergence of minimizers for $\kappa \to 0$]\label{cor:zeroconvergence}
		Let $\Lambda \in \prob(\setm)$ and for each $\kappa > 0$ let $\nu_\kappa \in \setm$ be a minimizer of $J_{\Lambda,\kappa}$. Then, there exists some $\nu_0 \in \setm$ such that $\nu_k \weakstar \nu_0$ as $\kappa \to 0$. In particular, $\nu_0$ is the unique minimizer of $J_{\Lambda,0}$.
	\end{corollary}
	
	\begin{proof}
		By Proposition \ref{prop:scalingLimits}, there exists a minimizer $\nu_0 \in \measp(\Omega)$ of $J_{\Lambda,0}$. Such a minimizer is indeed unique: assume this were not the case, so that there exists a second minimizer $\nu_0' \in \measp(\Omega)$. Fix any $\tau \in \measp(\Omega)$ such that $\nu_0, \nu_0' \ll \tau$ and define $v_0 = \diff \nu_0/\diff \tau$ and $v_0' = \diff \nu_0'/\diff \tau$.  Let $\bar{v} = \left( \tfrac{1}{2}\sqrt{v_0} + \tfrac12 \sqrt{v_0'}\right)^2$ and define $\bar{\nu} = \bar{v}\tau$ (note that this definition does not depend on the choice of $\tau$ by positive 1-homogeneity). For any $\mu \in \setm$, fix any $\tau_\mu \in \measp(\Omega)$ such that $\tau, \mu \ll \tau_\mu$ and, by strict convexity of $x \mapsto x^2$, compute
		\begin{align*}
			\Hell^2(\bar{\nu}, \mu) &= \int_\Omega \left( \sqrt{\frac{\diff \bar{\nu}}{\diff \tau_\mu}} - \sqrt{\frac{\diff \mu}{\diff \tau_\mu}} \right)^2 \,\diff\tau_\mu \\
			&= \int_\Omega \left( \frac{1}{2}\left(\sqrt{v_0\frac{\diff \tau}{\diff \tau_\mu}} - \sqrt{\frac{\diff \mu}{\diff \tau_\mu}}\right) +\frac{1}{2}\left(\sqrt{v_0' \frac{\diff \tau}{\diff \tau_\mu}} - \sqrt{\frac{\diff \mu}{\diff \tau_\mu}}\right) \right)^2 \,\diff\tau_\mu \\
			&< \frac{1}{2} \int_\Omega \left( \sqrt{v_0\frac{\diff \tau}{\diff \tau_\mu}} - \sqrt{\frac{\diff \mu}{\diff \tau_\mu}}\right)^2\,\diff\tau_\mu + \frac{1}{2} \int_\Omega \left(\sqrt{v_0' \frac{\diff \tau}{\diff \tau_\mu}} - \sqrt{\frac{\diff \mu}{\diff \tau_\mu}} \right)^2 \,\diff\tau_\mu \\
			&= \frac12 ( \Hell^2(\nu_0, \mu) + \Hell^2(\nu_0', \mu)).
		\end{align*}
		An integration in $\mu$ eventually provides $J_{\Lambda,0}(\bar{\nu}) < \frac12(J_{\Lambda,0}(\nu_0) + J_{\Lambda,0}(\nu_0'))$, which contradicts minimality of $\nu_0$ and $\nu_0'$ simultaneously and provides uniqueness of the minimizer of $J_{\Lambda,0}$.
		
		Let now $(\kappa_n)_n$ be any sequence converging to $0$ and $(\kappa_{n'})_{n'}$ be any subsequence. By Proposition \ref{prop:scalingLimits} there exists an additional subsequence $(\nu_{n''})_{n''}$ such that $\nu_{n''}$ converges weak* to a minimizer of $J_{\Lambda,0}$, hence it converges to $\nu_0$ by uniqueness. Since every subsequence of $(\kappa_n)_n$ admits a subsequence converging to the same limit $\nu_0$, we conclude the whole sequence $(\nu_n)_n$ converges to $\nu_0$. In turn, since any sequence $(\kappa_n)_n$ converging to $0$ admits the same limit $\nu_0$, the continuous limit as $\kappa \to 0$ follows.	
	\end{proof}
	
	\begin{remark}[Joint stability under changes of $\kappa$ and $\Lambda$]
		\label{rem:JointStability}
		In the case when $\kappa_\infty \in (0,\infty)$ the behaviour of the $\HK_\kappa$-barycenter w.r.t.~variations in $\kappa$ can be reduced to the study of variations in $\Lambda$ via Remark \ref{rem:HKKappaRescaling}, by working in some finitely re-scaled $\Omega/\kappa$, for a sufficiently small but finite $\kappa$ instead, and by relocating the mass of $\Lambda$ onto the re-scaled measures $\mu$. By applying the results from Section \ref{sec:ExistenceStability} one then finds that one can consider joint limits in $\Lambda$ and $\kappa$, and that the order in which the limits are taken does not matter.
		
		The situation is more intricate when $\kappa_\infty \in \{0,\infty\}$. In the latter case, it can be problematic when $\Lambda$ is not exclusively supported on measures of equal mass. In the former case one may obtain different cluster points of minimizers $\nu$, depending on the order or relative speed in which $\Lambda$ and $\kappa$ approach their limits. An example is given in Remark \ref{rem:KappaZeroLimitExample} further below. Combining the above results we find that in both cases one obtains the limit minimizer for $\Lambda_\infty$ and $\kappa_\infty$ by first going to the limit in $\Lambda$ and then in~$\kappa$.
	\end{remark}
	
	\subsection{Duality}
	\label{sec:Duality}
	We now show that a dual problem for \eqref{eq:primal} can be formulated as
	\begin{multline}\label{eq:dual}
		\sup \Bigg\{ 
		\int_\setm \int_\Omega \Psi(\mu,x) \,\diff\mu(x)\,\diff\Lambda(\mu)
		\,\Bigg|\, \Psi, \Phi \in \c(\setm \times \Omega), (\Psi(\mu,\cdot), \Phi(\mu,\cdot)) \in Q_\kappa \tn{ for all } \mu \in \setm, \\
		\text{ and } \int_\setm \Phi(\mu,y) \,\diff\Lambda(\mu) \geq 0 \text{ for all } y \in \Omega\Bigg\}
		\tag{$\mathcal{D}_{\Lambda,\kappa}$}
	\end{multline}
	We will study this duality in more detail in Section \ref{sec:Diracs} (including dual existence and primal-dual optimality conditions) for the specific case when $\Lambda$ is concentrated on the set of Dirac measures. For the general case we content ourselves with equality of optimal values.
	In \cite{friesecke2019barycenters} duality of \eqref{eq:primal} was established by combining all pairwise optimization problems \eqref{eq:HKSemiCoupling} for $\HK^2_\kappa(\mu_i,\nu)$ in the discrete version of \eqref{eq:primal} (with a finite collection of input measures $\mu_i$) and then dualizing them jointly. Here we generalize this combination to the case of uncountably many input measures.

	\begin{proposition}\label{prop:dualExistence}
		Let $\Lambda \in \prob(\setm)$ and $\kappa \in (0, \infty)$. Then, \eqref{eq:dual} is a dual problem to \eqref{eq:primal}, more precisely
		\begin{equation}\label{eq:strongduality}
			\text{\eqref{eq:dual} $=$ \eqref{eq:primal}}.
		\end{equation}
	\end{proposition}

	\begin{proof}
		For given $\Lambda \in \prob(\setm)$, define the measure $\Lambda \cdot \mu \in \measp(\setm \times \Omega)$ as
		\[
		\int_{\setm \times \Omega} \phi \,\diff(\Lambda \cdot \mu)
		\assign \int_{\setm} \int_{\Omega} \phi(\mu,x) \,\diff\mu(x) \,\diff\Lambda(\mu) \quad \tn{for } \phi \in \c(\setm \times \Omega).
		\]
		We start with the primal problem and estimate
		\begin{align}
			&\eqref{eq:primal} = \inf_{\nu \in \measp(\Omega)} \int_\setm \HK_{\kappa}^2(\mu, \nu) \,\diff \Lambda(\mu)
			\label{eq:steponeprimal}
			\\
			&\overset{\eqref{eq:HKSemiCoupling}}{=}
			\inf_{\nu \in \measp(\Omega)}
			\int_\setm
			\Bigg[
			\inf_{
				\substack{
					\gamma_1, \gamma_2, \gamma \in \measp(\Omega^2), \gamma_i \ll \gamma \\
					[(x, y) \mapsto x]_\#(\gamma_1) = \mu \\
					[(x, y) \mapsto y]_\#(\gamma_2) = \nu
				}
			}
			\int_{\Omega \times \Omega}
			c_\kappa\left( x, \tfrac{\diff\gamma_1}{\diff \gamma}(x,y), y, \tfrac{\diff\gamma_2}{\diff \gamma}(x,y) \right)
			\diff \gamma(x,y)
			\Bigg]
			\diff \Lambda(\mu)
			\label{eq:infinside}
			\\
			&\leq
			\inf_{\nu \in \measp(\Omega)}
			\inf_{
				\substack{
					\Gamma_1, \Gamma_2, \Gamma \in \measp(\setm \times \Omega^2), \Gamma_i \ll \Gamma \\
					[(\mu, x, y) \mapsto (\mu, x)]_\#(\Gamma_1) = \Lambda \cdot \mu \\
					[(\mu, x, y) \mapsto (\mu, y)]_\#(\Gamma_2) = \Lambda \otimes \nu
				}
			}
			\int_{\setm \times \Omega \times \Omega}
			c_\kappa\left( x, \tfrac{\diff\Gamma_1}{\diff \Gamma}(\mu,x,y), y, \tfrac{\diff\Gamma_2}{\diff \Gamma}(\mu,x,y) \right)
			\diff \Gamma(\mu,x,y)
			\label{eq:infoutside}
			\\
			&=
			\inf_{
				\substack{
					\Gamma_1, \Gamma_2, \Gamma \in \measp(\setm \times \Omega^2), \Gamma_i \ll \Gamma \\
					[(\mu, x, y) \mapsto (\mu, x)]_\#(\Gamma_1) = \Lambda \cdot \mu \\
					\exists \nu \in \measp(\Omega) \text{ s.t. } [(\mu, x, y) \mapsto (\mu, y)]_\#(\Gamma_2) = \Lambda \otimes \nu
				}
			}
			\int_{\setm \times \Omega \times \Omega}
			c_\kappa\left( x, \tfrac{\diff\Gamma_0}{\diff \Gamma}(\mu,x,y), y, \tfrac{\diff\Gamma_1}{\diff \Gamma}(\mu,x,y) \right)
			\diff \Gamma(\mu,x,y)
			\label{eq:extendedPrimal}
		\end{align}
		The inequality from \eqref{eq:infinside} to \eqref{eq:infoutside} follows since every admissible candidate in the latter induces a family of admissible candidates for the former. Indeed, let $\Gamma_1,\Gamma_2,\Gamma$ be admissible in \eqref{eq:infoutside}. By the constraints it follows that $[(\mu,x,y) \mapsto \mu]_\#(\Gamma_i) \ll \Lambda$. As in \eqref{eq:HKSemiCoupling}, since $c_\kappa$ is positively 1-homogeneous in its second and fourth argument, the value of the integral does not depend on the choice of $\Gamma$, as long as $\Gamma_i \ll \Gamma$. Therefore, w.l.o.g.~we may choose $\Gamma$ such that $[(\mu,x,y) \mapsto \mu]_\#(\Gamma) \ll \Lambda$. Let now $(\gamma_{1,\mu})_{\mu \in \setm},(\gamma_{2,\mu})_{\mu \in \setm}$ and $(\gamma_\mu)_{\mu \in \setm}$ be the disintegrations of $\Gamma_1$, $\Gamma_2$ and $\Gamma$ with respect to $\Lambda$. These three families of measures are then admissible in \eqref{eq:infinside} and yield the same score.

		Now set $X \assign \meas(\setm \times \Omega^2)$, $Y \assign \meas(\setm \times \Omega)$, and define
		\begin{align*}
			G & \colon X \times X \to \RCupInf, &
			(\Gamma_1, \Gamma_2) & \mapsto 
			\int_{\setm \times \Omega^2}
			c_\kappa\left( x, \tfrac{\diff\Gamma_1}{\diff \Gamma}, y, \tfrac{\diff\Gamma_2}{\diff \Gamma} \right)
			\diff \Gamma
			\\
			F_1 & \colon Y \to \RCupInf, &
			\tau & \mapsto \begin{cases} 0 & \tn{if } \tau = \Lambda \cdot \mu, \\
				+\infty & \tn{else.} \end{cases} \\
			F_2 & \colon Y \to \RCupInf, &
			\tau & \mapsto \begin{cases} 0 & \tn{if } \tau = \Lambda \otimes \nu \tn{ for some } \nu \in \measp(\Omega), \\
				+\infty & \tn{else.} \end{cases}
		\end{align*}
		where in the definition of $G$ the measure $\Gamma$ is any positive measure such that $\Gamma_1 \ll \Gamma$ and $\Gamma_2 \ll \Gamma$ (note that $G(\Gamma_1, \Gamma_2)$ is finite only if both $\Gamma_1$ and $\Gamma_2$ are non-negative). Let us also define the two linear (projection) operators $Q_1, Q_2 \colon X \to Y$ as
		\[
		Q_1 \Gamma \assign [(\mu, x, y) \mapsto (\mu, x)]_\#(\Gamma) \quad \tn{and}
		\quad
		Q_2 \Gamma \assign [(\mu, x, y) \mapsto (\mu, y)]_\#(\Gamma).
		\]
		Hence, we can rewrite \eqref{eq:extendedPrimal} as
		\begin{equation}
			\label{eq:abstarctP}
			\inf_{\Gamma_1, \Gamma_2\in X} G(\Gamma_1, \Gamma_2) + F_1(Q_1 \Gamma_1) + F_2(Q_2 \Gamma_2).
			\tag{$\mathcal{P}$}
		\end{equation}
		By standard convex duality theory one finds \eqref{eq:abstarctP} $\leq$ \eqref{eq:abstarctD}, where \eqref{eq:abstarctD} is
		\begin{equation}
			\label{eq:abstarctD}
			\sup_{\Psi, \Phi \in \c(\setm \times \Omega)} -G^*(Q_1^*\Psi, Q_2^*\Phi) - F_1^*(-\Psi) - F_2^*(-\Phi).
			\tag{$\mathcal{D}$}
		\end{equation}
		Note that we do not insist on a vanishing duality gap here.
		Direct computation quickly yields
		\begin{align*}
			F_1^* & :  \c(\setm \times \Omega) \to \RCupInf, & \Psi & \mapsto \int_{\setm \times \Omega} \Psi \,\diff (\Lambda \cdot \mu), \\
			F_2^* & :  \c(\setm \times \Omega) \to \RCupInf, & \Phi & \mapsto \begin{cases}
				0 & \tn{if } \int_{\setm} \Phi(\mu,y)\,\diff \Lambda(\mu) \leq 0 \;\;\forall\,y \in \Omega \\
				+ \infty &\tn{else.} \end{cases}
			\intertext{and using \cite[Lemma 2.9]{ChizatDynamicStatic2018}}
			G^*(Q_1^*\cdot, Q_2^*\cdot) & : \c(\setm \times \Omega)^2 \to \RCupInf, &
			(\Psi,\Phi) & \mapsto 
			\begin{cases} 0 & \tn{if } (\Psi(\mu,\cdot), \Phi(\mu,\cdot)) \in Q_\kappa \tn{ for all } \mu \in \setm, \\
				+ \infty &\tn{else.} \end{cases}
		\end{align*}
		With this, \eqref{eq:abstarctD} becomes
		\begin{multline*}
			\sup \Bigg\{ 
			\int_\setm \int_\Omega \Psi(\mu,x) \,\diff\mu(x)\,\diff\Lambda(\mu)
			\,\Bigg|\, \Psi, \Phi \in \c(\setm \times \Omega), (\Psi(\mu,\cdot), \Phi(\mu,\cdot)) \in Q_\kappa \tn{ for all } \mu \in \setm, \\
			\text{ and } \int_\setm \Phi(\mu,y) \,\diff\Lambda(\mu) \geq 0 \text{ for all } y \in \Omega\Bigg\},
		\end{multline*}
		which is exactly \eqref{eq:dual}. Let us now fix a minimizer $\nu_\kappa \in \measp(\Omega)$ of $J_{\Lambda, \kappa}$ and  continue from above
		\begin{align}
			\eqref{eq:primal} & \leq \eqref{eq:extendedPrimal} = \eqref{eq:abstarctP} \leq \eqref{eq:abstarctD} = \eqref{eq:dual}
			\\
			&=
			\sup_{
				\substack{
					\Psi, \Phi \in \c(\setm \times \Omega) \\
					(\Psi(\mu,\cdot), \Phi(\mu,\cdot)) \in Q_\kappa \; \forall \mu \in \setm \\
					\int_\setm \Phi(\mu,y) \,\diff\Lambda(\mu) \geq 0 \; \forall y \in \Omega
				}
			}
			\int_\setm \int_\Omega \Psi(\mu,x) \,\diff\mu(x)\,\diff\Lambda(\mu)
			\\
			&\leq
			\sup_{
				\substack{
					\Psi, \Phi \in \c(\setm \times \Omega) \\
					(\Psi(\mu,\cdot), \Phi(\mu,\cdot)) \in Q_\kappa \; \forall \mu \in \setm \\
					\int_\setm \Phi(\mu,y) \,\diff\Lambda(\mu) \geq 0 \; \forall y \in \Omega
				}
			}
			\int_\setm \int_\Omega \Psi(\mu,x) \,\diff\mu(x)\,\diff\Lambda(\mu) + \int_\Omega \int_\setm \Phi(\mu,y) \,\diff\Lambda(\mu) \,\diff\nu_\kappa(y)
			\\
			&\leq
			\sup_{
				\substack{
					\Psi, \Phi \in \c(\setm \times \Omega) \\
					(\Psi(\mu,\cdot), \Phi(\mu,\cdot)) \in Q_\kappa \; \forall \mu \in \setm
				}
			}
			\int_\setm \left[ \int_\Omega \Psi(\mu,x) \,\diff\mu(x) + \int_\Omega \Phi(\mu,y) \,\diff\nu_\kappa(y)\right]\,\diff\Lambda(\mu)
			\\
			&\leq
			\int_\setm \left[
			\sup_{
				\substack{
					\psi, \phi \in \c(\Omega) \\
					(\psi, \phi) \in Q_\kappa
				}
			}
			\int_\Omega \psi(x) \,\diff\mu(x) + \int_\Omega \phi(y) \,\diff\nu_\kappa(y)\right]\,\diff\Lambda(\mu)
			=
			\int_\setm \HK_\kappa^2(\mu, \nu_\kappa)\,\diff\Lambda(\mu) = \eqref{eq:primal}.
			\label{eq:laststep}
		\end{align}
		The chain of inequalities \eqref{eq:steponeprimal}-\eqref{eq:laststep} is then actually a chain of equalities. Hence, \eqref{eq:dual} is a dual problem to \eqref{eq:primal} and the optimal values coincide.
	\end{proof}

	\begin{remark}[Formal Wasserstein limit of \eqref{eq:dual}]
		Considering Theorem \ref{thm:HKScaling} and Proposition \ref{prop:scalingLimits} one might expect to recover a dual problem for the Wasserstein-2 distance by considering $\kappa^2 \cdot \text{\eqref{eq:dual}}$ and then sending $\kappa \to \infty$. The problem $\kappa^2 \cdot \text{\eqref{eq:dual}}$ can be written as
		\begin{multline*}
			\sup \Bigg\{ 
			\int_\setm \int_\Omega \Psi(\mu,x) \,\diff\mu(x)\,\diff\Lambda(\mu)
			\,\Bigg|\, \Psi, \Phi \in \c(\setm \times \Omega), (\Psi(\mu,\cdot)/\kappa^2, \Phi(\mu,\cdot)/\kappa^2) \in Q_\kappa \tn{ for all } \mu \in \setm, \\
			\text{ and } \int_\setm \Phi(\mu,y) \,\diff\Lambda(\mu) \geq 0 \text{ for all } y \in \Omega\Bigg\}.
		\end{multline*}
		At a purely intuitive level one can then consider the limit of the condition $(\Psi(\mu,\cdot)/\kappa^2, \Phi(\mu,\cdot)/\kappa^2) \in Q_\kappa$ for some $\mu \in \setm$ as $\kappa \to \infty$:
		\begin{multline*}
			1-\Psi(\mu,x)/\kappa^2-\Phi(\mu,y)/\kappa^2+o(1/\kappa^2) = (1-\Psi(\mu,x)/\kappa^2)(1-\Phi(\mu,y)/\kappa^2) \\
			\geq \Cos^2(|x-y|/\kappa) = 1-|x-y|^2/\kappa^2+o(1/\kappa^2).
		\end{multline*}
		That is, we expect to obtain the limit condition $\Psi(\mu,x)+\Phi(\mu,y) \leq |x-y|^2$, which would turn \eqref{eq:dual} into a version of the well-known dual problem for the Wasserstein barycenter. To the best of our knowledge this dual has so far not yet been stated in the literature for the case of a continuum of input measures.
	\end{remark}
	\section{Barycenter of a continuum of Dirac measures}
	\label{sec:Diracs}
	\subsection{Problem setup and basic properties}
	Throughout Section \ref{sec:Diracs} we study the particular case when $\Lambda$ is concentrated on the set of unit Dirac measures, i.e.~$\Lambda$-almost every $\mu$ is of the form $\delta_x$ for some $x \in \Omega$.
	In this case $\Lambda$ can be represented by a measure $\rho \in \prob(\Omega)$ which gives the distribution of the locations $x \in \Omega$. More precisely, for any $\rho \in \prob(\Omega)$ we define the measure $\Lambda_\rho \assign T_\# \rho$ where $T: \Omega \to \prob(\Omega)$, $x \mapsto \delta_x$, or equivalently
	\[
	\int_\setm \phi(\mu) \,\diff\Lambda_\rho(\mu) = \int_{\Omega} \phi(\delta_x) \,\diff\rho(x) \quad \text{for all } \phi \in \c(\setm).
	\]
	In this particular case the primal problem $(\mathcal{P}_{\Lambda_\rho, \kappa})$ simplifies to
	\begin{equation}
		\label{eq:primalr}
		\inf \left\{ J_{\rho,\kappa}(\nu) \assign \int_{\Omega} \HK_\kappa^2(\delta_x,\nu)\,\diff\rho(x) \middle| \nu \in \measp(\Omega) \right\}.
		\tag{$\mathcal{P}_{\rho,\kappa}$}
	\end{equation}
	For the Wasserstein case (i.e.~$\kappa=\infty$) this problem is trivial, the unique minimizer being given by $\nu=\delta_{\ol{x}}$ where $\ol{x}\assign \int_\Omega x\,\diff \rho(x)$ is the center of mass of $\rho$ ($\ol{x} \in \Omega$ by convexity of $\Omega$).
	For $\rho$ being a finite superposition of Dirac measures, i.e.~$\rho=\sum_{i=1}^n m_i\,\delta_{x_i}$, and $\kappa \in (0,\infty)$ the problem was studied in \cite{friesecke2019barycenters}.
	It was shown that for $\kappa$ sufficiently large the minimizer $\nu$ is again a single Dirac measure (consistent with the scaling limit of Proposition \ref{prop:scalingLimits}). However, for smaller $\kappa$, the minimizer $\nu$ may contain multiple Diracs or even be diffuse.
	
	Therefore, we now study \eqref{eq:primalr} in some more depth.
	First, we will further simplify the expression of \eqref{eq:primalr} by making the expression $\HK^2_\kappa(\delta_x,\nu)$ more explicit. Then, in Section \ref{sec:Dualityr} we revisit the dual problem, derive dual existence and primal-dual optimality conditions. In Section \ref{sec:DiscreteDiffuse} we present some results on whether barycenters $\nu$ are discrete or diffuse and we study the asymptotic behaviour of the barycenter as $\kappa \to 0$ in Section \ref{sec:Asymptotic}.
	
	\begin{proposition}\label{prop:HKdirac}
		Let $\kappa \in (0,\infty)$, $m > 0$, $\bar{x} \in \Omega$, $\nu \in \measp(\Omega)$. One finds
		\begin{align*}
			\HK_\kappa^2(m\delta_{\bar{x}}, \nu) &= \sup_{\xi < 1} \left[ m\xi + \|\nu\| - \frac{1}{1-\xi} \int_\Omega \Cos^2(|\bar{x}-y|/\kappa) \,\diff\nu(y) \right] \\
			&= m + \|\nu\| - 2\sqrt{m} \sqrt{ \int_\Omega \Cos^2(|\bar{x}-y|/\kappa) \,\diff\nu(y) }.
		\end{align*}
	\end{proposition}
	
	\begin{proof}
		Let $\mu = m \, \delta_{\bar{x}}$. We recall from \eqref{eq:HKDual} that
		\[
		\HK_\kappa^2(\mu, \nu) = \sup_{(\psi, \phi) \in Q_\kappa} \int_\Omega \psi(x)\,\diff \mu(x) + \int_\Omega \phi(y) \,\diff\nu(y) = \sup_{(\psi, \phi) \in Q_\kappa} m\psi(\bar{x}) + \int_\Omega \phi(y) \,\diff\nu(y),
		\]
		where
		\[
		Q_\kappa = \left\{(\psi,\phi)\in\c(\Omega)^2 \,\text{ s.t. }\,
		\begin{aligned}
			&\psi(x), \phi(y) \in (-\infty,1] \\ &(1-\psi(x))(1-\phi(y))\geq \Cos(|x-y|/\kappa)^2
		\end{aligned}
		\quad \forall x,y \in \Omega\right\}.
		\]
		Note that only the value of $\psi$ at $\bar{x}$ enters the energy. For any $\psi_{\bar{x}} \in \R$ and $n \in \N$ set $\psi_n(x) \assign \psi_{\bar{x}} - n \cdot \|x-\bar{x}\|$. For each $\psi_n$ the remaining supremum over $\phi$ is then attained by
		\[
		\phi_n(y) \assign \inf_{x \in \Omega} \;1 - \frac{\Cos^2(|x-y|/\kappa)}{1-\psi_n(x)},
		\]
		which is indeed a continuous function in $y$.
		As $n \to \infty$ one has $\phi_n \nearrow \phi$ pointwise for
		\[
		\phi(y) \assign 1 - \frac{\Cos^2(|\bar{x}-y|/\kappa)}{1-\psi_{\bar{x}}}
		\]
		i.e.~only the constraint for $x=\bar{x}$ in $Q_\kappa$ remains. Therefore, by monotone convergence, the problem reduces to
		\begin{align*}
			\HK_\kappa^2(m\delta_{\bar{x}}, \nu) &= \sup_{\psi_{\bar{x}} < 1} \left[ m\psi_{\bar{x}} + \|\nu\| - \frac{1}{1-\psi_{\bar{x}}} \int_\Omega \Cos^2(|\bar{x}-y|/\kappa) \,\diff\nu(y) \right] \\
			&= m + \|\nu\| - 2\sqrt{m} \sqrt{ \int_\Omega \Cos^2(|\bar{x}-y|/\kappa) \,\diff\nu(y) }.
			\qedhere
		\end{align*}
	\end{proof}
	
	\begin{corollary}
		A primal minimizer for $\HK_\kappa^2(m\delta_{\bar{x}}, \nu)$ in \eqref{eq:HKSoftMarginal} is given by
		
		\begin{equation}
			\label{eq:PrimalDiracMin}
			\gamma = \delta_{\bar{x}} \otimes \sigma
			\qquad \tn{with} \qquad
			\sigma = \nu \; \Cos^2(|\bar{x}-\cdot|/\kappa) \; \sqrt{\frac{m}{\|\nu \; \Cos^2(|\bar{x}-\cdot|/\kappa) \|}} \;.
		\end{equation}
	\end{corollary}
	\begin{proof}
		This follows directly by plugging expression  \eqref{eq:PrimalDiracMin} into \eqref{eq:HKSoftMarginal} and comparing the objective with Proposition~\ref{prop:HKdirac}.
	\end{proof}
	If we were to interpret an HK-barycenter $\nu$ as a `generalized clustering' of some input data $\rho$, then for each $\bar{x} \in \Omega$, the corresponding measure $\sigma$ (with $m=1$) could be interpreted as the association strength of the point at $\bar{x}$ with each of the points in the clustering. It would be a common occurrence that a point is associated with multiple `clusters' at the same time.
	
	Next, Proposition \ref{prop:HKdirac} also yields a simpler form of the primal objective which we will subsequently study in more detail.
	
	\begin{corollary}
		Let $\kappa \in (0,\infty)$ and $\rho \in \prob(\Omega)$. Then \eqref{eq:primalr} admits a minimizer $\nu \in \measp(\Omega)$ and the objective function $J_{\rho,\kappa}$ in \eqref{eq:primalr} takes the form
		\begin{align}
			\label{eq:primalrDirect}
			J_{\rho,\kappa}(\nu) = 1 + \|\nu\| -2 \int_{\Omega} \sqrt{\int_\Omega \Cos^2(|x-y|/\kappa)\,\diff\nu(y)}\,\diff\rho(x).
		\end{align}
	\end{corollary}
	\begin{proof}
		Existence of a minimizer follows by Proposition \ref{prop:primalexistence} applied for $\Lambda = \Lambda_\rho$, the simplified objective in \eqref{eq:primalrDirect} follows by applying Proposition \ref{prop:HKdirac} for the integrand $\HK_\kappa^2(\delta_x,\nu)$ inside $J_{\rho, \kappa}$ in \eqref{eq:primalr}.
	\end{proof}

	\subsection{Duality}
	\label{sec:Dualityr}
	Next, we prove that when $\Lambda = \Lambda_\rho$, the dual \eqref{eq:dual} takes the specific form
	
	\begin{multline}\label{eq:dualr}
		\sup \Bigg\{ 
		\int_{\Omega} \psi(x)\,\diff\rho(x)
		\;\Bigg|\; \psi \in \c(\Omega),\, \psi < 1 \\
		\tn{ and }	F_{\rho,\kappa}(\psi)(y) \assign \int_{\Omega} \frac{\Cos^2(|x-y|/\kappa)}{1-\psi(x)} \,\diff\rho(x) \leq 1 \text{ for all } y \in \Omega\Bigg\}.
		\tag{$\mathcal{D}_{\rho,\kappa}$}
	\end{multline}
	
	In the following, we will refer to $F_{\rho,\kappa}$ as the \textit{constraint function}.

	\begin{proposition}
		\label{prop:DualityR}
		Let $\rho \in \prob(\Omega)$ and $\kappa \in (0, \infty)$. Then,
		\begin{enumerate}[(i)]
			\item \eqref{eq:dualr} is a dual problem to \eqref{eq:primalr} and \eqref{eq:dualr} $ = $ \eqref{eq:primalr},
			\item \eqref{eq:dualr} admits a maximizer $\psi \in \c(\Omega)$ which is unique on the support of $\rho$ and for any primal optimizer $\nu$ we have
			\begin{subequations}\label{eq:pdEquation}
				\begin{align}
					\label{eq:pdEquationPsi}
					\psi(x) & = 1- \sqrt{\int_\Omega \Cos^2(|x-y|/\kappa)\,\diff\nu(y)}
					&\text{ for $\rho$-a.e.~} x \in \Omega, \\
					\label{eq:pdEquationActive}
					F_{\rho,\kappa}(\psi)(y)&=1
					&\text{ for $\nu$-a.e.~} y \in \Omega,
				\end{align}		
			\end{subequations}
			\item \label{point:optimalPair} 
			an admissible couple $(\nu, \psi) \in \measp(\Omega) \times \c(\Omega)$ is optimal if and only if \eqref{eq:pdEquation} holds.
		\end{enumerate}
	\end{proposition}
	
	\begin{proof}
		\emph{Step 1. Primal optimality conditions.} Let $\nu \in \measp(\Omega)$ be an optimizer of \eqref{eq:primalr} and consider any non-negative measure $\tilde\nu \in \measp(\Omega)$. By optimality of $\nu$, one has
		\[
		\frac{\diff}{\diff t} J_{\rho, \kappa}(\nu + t\tilde\nu) |_{t=0^+} \geq 0.
		\]
		Using \eqref{eq:primalrDirect}, we compute
		\begin{align}
			0 &\leq
			\frac{\diff}{\diff t} J_{\rho, \kappa}(\nu + t\tilde\nu)|_{t=0^+}
			= 
			\frac{\diff}{\diff t} \left.\left( 1 + \|\nu\| + t\|\tilde\nu\| -2 \int_{\Omega} \sqrt{\int_\Omega \Cos^2(|x-y|/\kappa)\,\diff(\nu + t\tilde\nu)(y)}\,\diff\rho(x) \right)\right|_{t=0^+}
			\nonumber
			\\
			&= \int_\Omega \left[ 1 - \int_\Omega \frac{\Cos^2(|x-y|/\kappa)}{\sqrt{\int_\Omega \Cos^2(|x-z|/\kappa)\,\diff\nu(z)}}\,\diff\rho(x) \right] \,\diff \tilde\nu(y).
			\nonumber
		\end{align}
		Since $\tilde\nu$ is an arbitrary non-negative measure, we conclude that, for any optimal $\nu$,
		\begin{equation}\label{eq:optimalityI}
			\int_\Omega \frac{\Cos^2(|x-y|/\kappa)}{\sqrt{\int_\Omega \Cos^2(|x-z|/\kappa)\,\diff\nu(z)}}\,\diff\rho(x) \leq 1 \quad \tn{for all } y \in \Omega.
		\end{equation}
		Now consider as variation $\tilde\nu=-\nu$ such that $\nu+t \tilde\nu \geq 0$ for $t \in[0,1]$. As above, we find
		\begin{align*}
			0 \leq
			\frac{\diff}{\diff t} J_{\rho, \kappa}(\nu + t\tilde\nu)|_{t=0^+}
			= 
			-\int_\Omega \left[ 1 - \int_\Omega \frac{\Cos^2(|x-y|/\kappa)}{\sqrt{\int_\Omega \Cos^2(|x-z|/\kappa)\,\diff\nu(z)}}\,\diff\rho(x) \right] \,\diff \nu(y).
		\end{align*}
		Since from above we know that the expression in squared brackets must be non-negative, we now deduce that
		\begin{equation}\label{eq:optimalityII}
			\int_\Omega \frac{\Cos^2(|x-y|/\kappa)}{\sqrt{\int_\Omega \Cos^2(|x-z|/\kappa)\,\diff\nu(z)}}\,\diff\rho(x) = 1 \quad \tn{for }\nu\tn{-a.e. } y \in \Omega.
		\end{equation}
		
		\medskip
		\noindent
		\emph{Step 2. Duality}. By Proposition \ref{prop:dualExistence} as applied to $\Lambda_\rho$, we have \eqref{eq:primalr} $ = (\mathcal{P}_{\Lambda_\rho, \kappa}) = (\mathcal{D}_{\Lambda_\rho, \kappa})$. The latter problem, taking into account the definition of $\Lambda_\rho$, reduces to
		\begin{multline}
			\label{eq:lambdarhodual}
			\sup \Bigg\{ 
			\int_\Omega \Psi(\delta_x,x) \,\diff\rho(x)
			\,\Bigg|\, \Psi, \Phi \in \c(\setm \times \Omega), (\Psi(\mu,\cdot), \Phi(\mu,\cdot)) \in Q_\kappa \tn{ for all } \mu \in \setm, \\
			\text{ and } \int_\Omega \Phi(\delta_x,y) \,\diff\rho(x) \geq 0 \text{ for all } y \in \Omega\Bigg\}
			\tag{$\mathcal{D}_{\Lambda_\rho,\kappa}$}
		\end{multline}
		We now show that \eqref{eq:lambdarhodual} $ \leq $ \eqref{eq:dualr}. Let $\Psi, \Phi \in \c(\setm \times \Omega)$ be any two admissible functions for \eqref{eq:lambdarhodual}. Define $\psi\colon \Omega \to \R$ as $\psi(x) \assign \Psi(\delta_x, x)$. Then $\psi \in \c(\Omega)$. Since $(\Psi(\mu,\cdot), \Phi(\mu,\cdot)) \in Q_\kappa$ for all $\mu \in \setm$, one has $\psi(x) = \Psi(\delta_x,x) < 1$ for all $x \in \Omega$ and in particular
		\[
		(1-\Psi(\delta_x,x))(1-\Phi(\delta_x,y)) \geq \Cos(|x-y|/\kappa) \quad \tn{for all } x,y \in \Omega,
		\]
		so that
		\[
		\frac{\Cos(|x-y|/\kappa)}{1-\psi(x)} \leq 1 - \Phi(\delta_x,y) \tn{ for all } x,y \in \Omega.
		\]
		Integrating against $\rho$ and using that $\|\rho\| = 1$, we get
		\[
		F_{\rho,\kappa}(\psi)(y) = \int_\Omega \frac{\Cos(|x-y|/\kappa)}{1-\psi(x)} \,\diff\rho(x) \leq 1 - \int_\Omega \Phi(\delta_x,y) \,\diff\rho(x) \leq 1 \quad \tn{for all } y \in \Omega.
		\]
		Hence $\psi$ is admissible for \eqref{eq:dualr} and \eqref{eq:lambdarhodual} $ \leq $ \eqref{eq:dualr}. Let now $\psi \in \c(\Omega)$, $\psi < 1$, be admissible for \eqref{eq:dualr} and let $\nu \in \measp(\Omega)$ be any optimizer of \eqref{eq:primalr}. Then, since $F_{\rho, \kappa}(\psi) \leq 1$, one obtains
		\begin{align}
			\int_\Omega \psi(x)\,\diff \rho(x)
			\nonumber 
			&\leq
			\int_\Omega \psi(x)\,\diff \rho(x) + \int_\Omega \left[ 1 - \int_{\Omega} \frac{\Cos^2(|x-y|/\kappa)}{1-\psi(x)} \,\diff\rho(x)  \right]\,\diff\nu(y)
			\nonumber \\
			&= \int_\Omega \left[ \psi(x) + \|\nu\| - \frac{1}{1-\psi(x)} \int_{\Omega} \Cos^2(|x-y|/\kappa) \,\diff\nu(y)  \right]\,\diff\rho(x)
			\nonumber \\
			&\leq
			\int_\Omega \sup_{\xi < 1} \left[ \xi + \|\nu\| - \frac{1}{1-\xi} \int_{\Omega} \Cos^2(|x-y|/\kappa) \,\diff\nu(y)  \right]\,\diff\rho(x)
			\nonumber \\
			&=
			\int_\Omega \HK_\kappa^2(\delta_x, \nu)\,\diff\rho(x) = \eqref{eq:primalr},
			\label{eq:inequalityIIpsi}
		\end{align}
		where the first equality in the last line follows by Proposition \ref{prop:HKdirac}. All in all, we showed \eqref{eq:primalr} $ = (\mathcal{P}_{\Lambda_\rho, \kappa}) = (\mathcal{D}_{\Lambda_\rho, \kappa}) \leq $ \eqref{eq:dualr} $ \leq $ \eqref{eq:primalr}, hence \eqref{eq:dualr} is a dual to \eqref{eq:primalr} and \eqref{eq:primalr} $ = $ \eqref{eq:dualr}.
		
		\medskip
		\noindent
		\emph{Step 3. Existence and characterization of the dual optimizer.} Fix a primal optimizer $\nu_\kappa$ and define $\psi_\kappa \in \c(\Omega)$ as
		\begin{equation}\label{eq:pdoptimalityLocal}
			(1-\psi_\kappa(x))^2 = \int_\Omega \Cos^2(|x-y|/\kappa)\,\diff\nu_\kappa(y).
		\end{equation}
		By the optimality condition \eqref{eq:optimalityI} we have $F_{\rho, \kappa}(\psi_\kappa) \leq 1$, so that $\psi_\kappa$ is dual admissible. Furthermore, thanks to \eqref{eq:optimalityII}, we have $F_{\rho, \kappa}(\psi_\kappa) = 1$ $\nu_\kappa$-almost everywhere. Thus, for $\nu = \nu_\kappa$ and $\psi = \psi_\kappa$, the chain of inequalities \eqref{eq:inequalityIIpsi} becomes a chain of equalities and $\psi_\kappa$ provides an optimizer for \eqref{eq:dualr}. In particular, from \eqref{eq:inequalityIIpsi} we also deduce that
		\begin{itemize}
			\item any optimal $\psi$ has to satisfy
			\[
			(1-\psi(x))^2 = \int_\Omega \Cos^2(|x-y|/\kappa)\,\diff\nu(y) \quad \tn{ for } \rho \tn{-a.e. } x \in \Omega, \tn{ for any primal optimizer } \nu,
			\]
			in order to have an equality between the second and third line (and recall that feasible $\psi$ must be $< 1$),
			\item for any optimal $\psi$ we have
			\[
			F_{\rho, \kappa}(\psi)(y) = 1 \tn{ for } \nu\tn{-a.e. } y \in \Omega, \tn{ for any primal optimizer } \nu,
			\]
			in order to have an equality in the first line.
		\end{itemize}
		Hence, \emph{(ii)} follows. Assume now $(\nu, \psi) \in \measp(\Omega) \times \c(\Omega)$ is an admissible couple such that \eqref{eq:pdEquation} holds. The same derivation as in \eqref{eq:inequalityIIpsi} provides
		\[
		\int_\Omega \psi(x)\,\diff\rho(x) = \int_\Omega \HK_\kappa^2(\delta_x, \nu)\,\diff\rho(x),
		\]
		which holds if and only if $\nu$ is a primal minimizer and $\psi$ is a dual maximizer, thus completing the proof.
	\end{proof}

	\begin{corollary}
		The constraint functions $F_{\rho,\kappa}(\psi)$ for any dual optimizer $\psi$ are identical.
	\end{corollary}
	\begin{proof}
		This follows from the fact that all dual maximizers agree $\rho$-a.e., see \eqref{eq:pdEquationPsi}, and the constraint function only evaluates $\psi$ on this set.
	\end{proof}
	
	Similar to the primal case (see Sections \ref{sec:ExistenceStability} and \ref{sec:Scaling}) dual maximizers and the constraint function are stable under small perturbations in $\rho$ and $\kappa$.
	\begin{proposition}[Dual stability]
		\label{prop:DualStability}
		Let $(\kappa_n)_n$ and $(\rho_n)_n$ be convergent (weak* in the latter case) sequences in $(0,\infty)$ and $\prob(\Omega)$, with limits $\kappa_\infty \in (0,\infty)$ and $\rho_\infty \in \prob(\Omega)$, respectively.
		Let $(\nu_n)_n$ be a corresponding sequence of primal minimizers and set
		\begin{align*}
			\psi_n(x) & \assign 1- \sqrt{\int_\Omega \Cos^2(|x-y|/\kappa_n)\,\diff \nu_n(y)}, &
			\psi_\infty(x) & \assign 1- \sqrt{\int_\Omega \Cos^2(|x-y|/\kappa_\infty)\,\diff \nu_\infty(y)},
		\end{align*}
		where $\nu_\infty$ is some cluster point of $(\nu_n)_n$.
		Then the latter is a dual maximizer of the limit problem for $\kappa_\infty$ and $\rho_\infty$, the sequence $(\psi_n)_n$ converges uniformly to $\psi_\infty$ on the support of $\rho_\infty$, and the sequence of constraint functions $(F_{\rho_n,\kappa_n}(\psi_n))_n$ converges uniformly to $F_{\rho_\infty,\kappa_\infty}(\psi_\infty)$ on $\Omega$.
	\end{proposition}
	While the primal minimizer might not always be unique, the set where $F_{\rho,\kappa}(\psi)=1$ for a dual maximizer $\psi$ is unique and stable under small perturbations in $\rho$ and $\kappa$. Since one must have $F_{\rho,\kappa}(\psi)(y)=1$ for $\nu$-almost all $y$, the set where $F_{\rho,\kappa}(\psi)$ is (close to) 1 therefore provides an alternative and unique interpretation of clustering.
	
	\begin{proof}
		By Corollary \ref{cor:minimizersconvergence} the sequence $(\nu_n)_n$ is weak* pre-compact and any cluster point is a minimizer of the primal limit problem, see Remark \ref{rem:JointStability} for the incorporation of a sequence of changing $(\kappa_n)_n$, with limit in $(0,\infty)$. By Proposition \ref{prop:DualityR} a dual maximizer for the limit problem is then given through \eqref{eq:pdEquationPsi}, which gives dual optimality of $\psi_\infty$.
		
		Since the family of functions $(y \mapsto \Cos^2(|x-y|/\kappa_n))_{x \in \Omega,n}$ is uniformly equicontinuous, so are the $(\psi_n)_n$ and one has that the subsequence of functions $(\psi_{n_k})_k$ is uniformly convergent for every weak* convergent subsequence $(\nu_{n_k})_k$ of $(\nu_n)_n$ and each limit must be a dual maximizer. Since the limit dual maximizer is unique on the support of $\rho_\infty$, all cluster points of $(\psi_n)_n$ must agree on this set (and no other cluster points can exist, e.g.~since all cluster points $(\nu_n)_n$ are primal minimizers).
		
		Finally, let us consider the sequence of constraint functions.
		For brevity set $F_n \assign F_{\rho_n,\kappa_n}(\psi_n)$. By assumption the sequence $(\kappa_n)_n$ is bounded away from zero, so there exists a finite set $Y \subset \Omega$ such that $\sum_{y \in Y} \Cos(|x-y|^2/\kappa_n) \geq 1$ for all $x \in \Omega$ and $n$. Then, with $F_n(y) \leq 1$ for all $y \in \Omega$, $n \in \N$ it follows that
		\begin{align*}
			\int_\Omega \frac{1}{1-\psi_n(x)}\diff \rho_n(x) \leq 
			\sum_{y \in Y} \int_\Omega \frac{\Cos(|x-y|^2/\kappa_n)}{1-\psi_n(x)}\diff \rho_n(x)
			= \sum_{y \in Y} F_n(y) \leq |Y|.
		\end{align*}
		$((1-\psi_n)^{-1} \cdot \rho_n)_n$ is therefore a sequence of bounded non-negative measures on $\Omega$ and thus weak* pre-compact. Again, by equicontinuity of the $(y \mapsto \Cos^2(|x-y|/\kappa_n))_{x,n}$ follows the pre-compactness of the sequence $(F_n)_n$ for the uniform convergence. Let now $(n_k)_k$ be a subsequence such that $((1-\psi_{n_k})^{-1} \cdot \rho_{n_k})_k$ converges weak*, let $F_\infty$ be the limit of $(F_{n_k})_k$. We find that $F_{n_k}(y)$ converges pointwise to $F_{\rho_\infty,\kappa_\infty}(\psi_\infty)(y)$ for all $y$ and thus we must have that $F_\infty=F_{\rho_\infty,\kappa_\infty}(\psi_\infty)$. Since the latter only depends on the value of $\psi_\infty$ on the support of $\rho_\infty$ and $\psi_\infty$ is unique on this support, this means that all cluster points $F_\infty$ must be identical.
	\end{proof}
	
	\subsection{Discrete and diffuse barycenters}
	\label{sec:DiscreteDiffuse}
	
	In \cite{friesecke2019barycenters} it was observed that the HK barycenter between a finite number of Dirac measures was sometimes discrete and sometimes diffuse. In the latter case it was shown that the solution is non-unique and that a discrete solution also exists \cite[Proposition 6.2]{friesecke2019barycenters}.
	In this Section we give an alternative proof for this result.
	Then we turn to the question of existence of discrete solutions for a diffuse $\rho$ and provide a negative answer: sometimes no discrete minimizers exist. For illustration we also briefly discuss an example on the torus.
	
	\begin{proposition}[Discrete barycenters for finite number of Dirac input measures]
		\label{prop:DiscreteBar}
		Let $\rho\assign \sum_{i=1}^n m_i\,\delta_{x_i}$ for $n \in \N$, $m = (m_1,\ldots,m_n) \in \R_+^n$, $\sum_{i=1}^n m_i=1$ and $x_1,\ldots,x_n \in \Omega$. Then, \eqref{eq:primalr} has a minimizer $\nu$ of the form
		\begin{align*}
			\nu = \sum_{i=1}^k \tilde{m}_i\,\delta_{\tilde{x}_i}
		\end{align*}
		for a positive integer $k \leq n$, non-negative mass coefficients $\tilde{m}_1,\ldots,\tilde{m}_k$ and positions $\tilde{x}_1,\ldots,\tilde{x}_k \in \Omega$.
	\end{proposition}
	\begin{proof}
		Let $\nu \in \measp(\Omega)$ be a minimizer of \eqref{eq:primalr} and let $\psi \in \c(\Omega)$ be the optimal dual defined via \eqref{eq:pdEquation}. Since $\psi$ is uniquely determined on $\spt \rho$, we can reduce our focus to a vector $\psi = (\psi_1,\dots,\psi_n) \in \R^n$ with entries defined as
		\[
		\psi_i = \int_\Omega \Cos^2(|x_i - y|/\kappa)\,\diff\nu(y) \quad \tn{for all } i = 1,\dots,n.
		\]
		Consider now a discrete approximating sequence $\{\nu^s\}_{s \in \N}$ for $\nu$, $\spt \nu^s \subset \spt \nu$, with
		\[
		\nu^s = \sum_{j=1}^s \bar{m}_j^s \delta_{x_j^s} \quad \tn{ for } \bar{m}^s = (\bar{m}_1^s,\ldots,\bar{m}_s^s) \in \R_+^s, \; x_1^s,\ldots,x_s^s \in \Omega,
		\]
		such that $\nu^s \weakstar \nu$ as $s \to \infty$. Each measure $\nu^s$ defines a vector $\psi^s \in \R^n$ by setting
		\[
		\psi_i^s \assign \int_\Omega \Cos^2(|x_i - y|/\kappa)\,\diff\nu^s(y)
		= \sum_{j=1}^s \Cos^2(|x_i - x_j^s|/\kappa)\,\bar{m}_j^s
		\quad \tn{ for all } i = 1,\dots,n.
		\]
		This can be written as $\psi^s=A^s \bar{m}^s$
		for a matrix $A^s \in \R^{n \times s}$ with entries $A^s_{ij} \assign \Cos^2(|x_i - x_j^s|/\kappa)$. Clearly we have $\textup{rank}(A^s) \leq n$, and so we can find a vector $\bar{m}^{s,n} \in \R^s_+$ with at most $n$ strictly positive entries such that $\psi^s = A^s \cdot \bar{m}^{s,n}$. In turn, this defines a discrete non-negative measure $\nu^{s,n}$ supported on at most $n$ points such that
		\[
		\psi_i^s = \int_\Omega \Cos^2(|x_i - y|/\kappa)\,\diff\nu^{s,n}(y) \quad \tn{for all } i = 1,\dots,n.
		\]
		By compactness, there exists a cluster point $\nu^n \in \measp(\Omega)$ such that, up to selection of a subsequence, $\nu^{s,n} \weakstar \nu^n$ and $\nu^n$ is supported on at most $n$ points because each measure $\nu^{s,n}$ is. Hence, using that $\nu^s \weakstar \nu$ and $\nu^{s,n} \weakstar \nu^n$ as $s \to \infty$ and that $A^s \bar{m}^{s} = A^s \bar{m}^{s,n}$, we obtain
		\begin{align*}
			\psi_i &= \int_\Omega \Cos^2(|x_i - y|/\kappa)\,\diff\nu(y) = \lim_{s \to \infty} \int_\Omega \Cos^2(|x_i - y|/\kappa)\,\diff\nu^{s}(y) \\
			&= \lim_{s \to \infty} \int_\Omega \Cos^2(|x_i - y|/\kappa)\,\diff\nu^{s,n}(y) = \int_\Omega \Cos^2(|x_i - y|/\kappa)\,\diff\nu^{n}(y).
		\end{align*}
		Since $F_{\rho, \kappa}(\psi)(y) = 1$ for every $y \in \spt\nu$, we also have $F_{\rho, \kappa}(\psi)(x_j^s) = 1$ for every $j = 1,\dots,s$, and any $s > 0$. In particular, when passing to the limit as $s \to \infty$, one observes that $F_{\rho, \kappa}(\psi)(y) = 1$ for every $y \in \spt\nu^n$. By Proposition \ref{prop:DualityR}, point \eqref{point:optimalPair}, $\nu^n$ is primal optimal and the result follows.
	\end{proof}

	In this manuscript $\Omega$ is a subset of $\R^d$ and equipped with the Euclidean distance. The Hellinger--Kantorovich distance can be defined for non-negative measures over more general metric spaces \cite{liero2018optimal} and in particular it can be shown that the barycenter problem between Dirac measures can be extended to the $d$-torus $\Tor^d=\R^d/\Z^d$. One merely has to replace any occurrence of $\Omega$ by $\Tor^d$ and the Euclidean distance $|x-y|$ by the geodesic distance $\dist$ on $\Tor^d$. We now show that on the torus it may happen that no discrete minimizer $\nu$ exists.
	
	\begin{proposition}[Diffuse barycenters on the torus]
		Let $d \in \N$, $\Tor^d\assign\R^d/\Z^d$ be the $d$-dimensional unit torus (with circumference 1 along each dimension), equipped with its geodesic distance $\dist$ and let $\rho \in \prob(\Tor^d)$ be the uniform probability measure on $\Tor^d$. Then for $d>1$ there is no discrete barycenter $\nu$. For $d=1$ there is no discrete barycenter when $\kappa \cdot \pi$ is irrational or $\kappa \cdot \pi > 1$, otherwise discrete barycenters exist.
	\end{proposition}
	\begin{proof}
		Adapting Proposition \ref{prop:DualityR} we obtain existence of a dual maximizer $\psi \in \c(\Tor^d)$, which is unique $\rho$-a.e., i.e.~it is unique since $\rho$ has full support. It is characterized by
		\begin{align}
			\label{eq:PsiTorus}
			(1-\psi(x))^2= \int_{\Tor^d} \Cos^2(\dist(x,y)/\kappa)\,\diff \nu(y)
		\end{align}
		and the condition $\psi(x)<1$, where $\nu$ is an arbitrary primal minimizer.
		By symmetry and convexity of the problem, $\psi$ must be translation invariant, i.e.~it must be constant, and therefore we have $\psi \in C^\infty(\Tor^d)$.
		
		Assume now $d>1$.
		Note that $g_y :  x \mapsto \Cos^2(\dist(x,y)/\kappa)$ for some fixed $y \in \Tor^d$ is merely $C^1$ when $\kappa \pi/2 \leq 1/2$ and even only $C^0$ otherwise, due to its behaviour on the sphere of radius $\kappa \pi/2$ around $y$ or on the cut locus of $\dist(\cdot,y)$. Also, it is not possible to `cancel' these irregularities by carefully combining a countable number of $g_y$ for different $y$ with positive weights that have a finite sum.
		Therefore, $\psi$ cannot be constructed from a discrete $\nu$ via \eqref{eq:PsiTorus}.
		
		Now let $d=1$. For $\kappa \pi >1$ the function $g_y$ is merely $C^0$ and a finite number of $g_y$ cannot be combined into a smoother function and thus, as above, no discrete barycenter can exist.
		Let now $\kappa \pi \leq 1$ and in addition $\kappa \pi \in \Q$. Then for any $y \in \Tor^1$ the set
		\begin{align*}
			S_y \assign \{y + k \cdot \kappa \pi\,|\,k \in \Z\}
		\end{align*}
		with the obvious interpretation of addition on the torus and identification of points that differ by an integer, is finite. Then by setting $\nu \assign m \sum_{y' \in S_y} \delta_{y'}$ with a suitable $m>0$ (depending on $\kappa$), one will find that it is possible to construct a constant $\psi$ via \eqref{eq:PsiTorus} and hence this $\nu$ is a discrete primal minimizer.
		For $\kappa \pi \notin \Q$ this construction fails since $S_y$ will not be finite and therefore no countable number of $g_y$ for different $y$ with positive weights that have a finite sum yields a $C^\infty$-function.
	\end{proof}
	
	From this example we draw the following intuition for $\R^d$: When $\rho$ has a large (compared to $\kappa$) region of constant density, the question whether the barycenter $\nu$ can be discrete or diffuse it not decided in the bulk of the region but at its boundary. Therefore, by careful design of the boundary region it might be possible to construct $\rho$ for which no discrete barycenter exists. We confirm this in the next proposition.
	
	\begin{proposition}[Diffuse barycenters in $\R^d$]
		Let $d \in \N$ and $\kappa \in (0, \infty)$. There exist a compact, closed, convex set $\Omega \subset \R^d$ and a measure $\rho \in \measp(\Omega)$ such that \primalr has no discrete optimizer.
	\end{proposition}
	
	\begin{proof}[Proof]
		Let $L \in (0,\infty)$ and let $\Omega \assign \bar{B}(0, L+\kappa\pi/2) \subset \R^d$ be the closed ball centered at the origin with radius $L+\kappa\pi/2$. Define the function $\sigma \in \c(\Omega)$ as
		\[
		\sigma^2(x) \assign \int_{B(0,L)} \Cos^2(|x-y|/\kappa)\,\diff y \quad \tn{for all } x \in \Omega.
		\]
		Denote $C_d \assign \|\Cos^2(|\cdot|)\|_{L^1(\R^d)}$ and $a \assign 1/\|\sigma\|_{L^1(\Omega)}$, and consider the probability measure $\rho \assign a\cdot \sigma \cdot \Lebesgue^d$. Let $(\nu, \psi) \in \measp(\Omega) \times \c(\Omega)$ be defined as
		\[
		\nu \assign C_d^2 \kappa^{2d} a^2 \cdot \Lebesgue^d\restr B(0, L)
		\quad \tn{and} \quad
		\psi \assign 1 - C_d\kappa^d\cdot a \cdot \sigma.
		\]
		The pair $(\nu, \psi) \in \measp(\Omega) \times \c(\Omega)$
		is an optimal primal-dual pair for \primalr and \dualr. Indeed, one readily checks that
		\[
		F_{\rho, \kappa}(\psi)(y) = \int_{\Omega} \frac{\Cos^2(|x-y|/\kappa)}{1-\psi(x)} \diff \rho(x)= \frac{1}{C_d\kappa^d} \int_{\Omega} \Cos^2(|x-y|/\kappa)\,\diff x
		\leq 1 \quad \tn{for all } y \in \Omega
		\]
		and $F_{\rho, \kappa}(\psi)(y) = 1$ for all $y \in \bar{B}(0,L) = \spt \nu$. Further, by construction,
		\[
		(1-\psi(x))^2 = C_d^2 \kappa^{2d} a^2 \sigma^2(x) =
		C_d^2 \kappa^{2d} a^2 \int_{B(0,L)} \Cos^2(|x-y|/\kappa)\,\diff y = \int_{\Omega} \Cos^2(|x-y|/\kappa)\,\diff \nu(x).
		\]
		Hence, optimality of $\nu$ and $\psi$ follows from Proposition \ref{prop:DualityR}, point \eqref{point:optimalPair}.
		Since $\psi$ is unique on $\spt \rho$, the function $F_{\rho, \kappa}(\psi)$ is unique and identical for all dual maximizers, therefore the set $\{y \in \Omega | F_{\rho, \kappa}(\psi)(y) = 1\}=\bar{B}(0,L)$ is unique, and finally we find that any primal minimizer must be concentrated on $\bar{B}(0,L)$.
		
		Denote by $\Cos^2_\kappa : \R^d \to [0,1]$ the function $x \mapsto \Cos^2(|x|/\kappa)$. Extend the measure $\nu$ from $\Omega$ to $\R^d$ by zero. Then by the above, on obtains for the convolution
		\begin{align*}
			(\Cos^2_\kappa * \nu)(y) \assign \int_{\R^d} \Cos^2_\kappa(y-x)\,\diff \nu(x)
			= \begin{cases}
				\sigma^2(y) & \tn{for } y \in \Omega, \\
				0 & \tn{else,}
			\end{cases}
		\end{align*}
		that is, it is known on all of $\R^d$. Now the proof strategy is to show that since the convolution of $\nu$ with a compact kernel is fully known, $\nu$ must indeed be uniquely determined and be equal to the above, and hence no other primal minimizer exists, in particular none that is discrete.
		
		\newcommand{\FT}{\mc{F}}
		\newcommand{\FTT}{\hat{\mc{F}}}
		
		Denote by $\FT$ the Fourier transform on $\R^d$, acting on suitable functions $f$ and measures $\mu$ as
		\begin{align*}
			(\FT f)(k) & \assign \frac{1}{(2\pi)^d} \int_{\R^d} f(x)\,\exp(ikx)\,\diff x, &
			(\FT \mu)(k) & \assign \frac{1}{(2\pi)^d} \int_{\R^d} \exp(ikx)\,\diff \mu(x)
		\end{align*}
		whenever these integrals are well-defined.
		Since $(\Cos^2_\kappa * \nu) \in L^2(\R^d)$, $\FT (\Cos^2_\kappa * \nu)$ is well-defined.
		The convolution theorem now corresponds to the observation that for almost every $k$ one has
		\begin{align*}
			(\FT \Cos^2_\kappa * \nu)(k) & = \frac{1}{(2\pi)^d} \int_{\R^d} (\Cos^2_\kappa * \nu)(x)\exp(ikx)\diff x \\
			& = \frac{1}{(2\pi)^d} \int_{\R^d} \int_{\R^d} \Cos^2_\kappa(x-y) \diff \nu(y) \exp(ikx)\diff x
			= \int_{\R^d} (\FT \Cos^2_\kappa)(k) \exp(iky) \diff \nu(y) \\
			& = (2\pi)^d (\FT \Cos^2_\kappa)(k)\,(\FT \nu)(k),
		\end{align*}
		where we swapped the order of integration by Fubini's theorem.
		
		Since $\Cos^2_\kappa$ has compact support, $\FT \Cos^2_\kappa(k) \neq 0$ $k$-almost everywhere and thus we obtain that
		\begin{align*}
			(\FT \nu)(k) = \frac{(\FT \Cos^2_\kappa * \nu)(k)}{(2\pi)^d (\FT \Cos^2_\kappa)(k)}
		\end{align*}
		for almost all $k$, i.e.~the Fourier transform of all primal minimizers must agree almost everywhere.
		
		We now show that for a finite measure $\mu$ on $\R^d$ with compact support one finds [$\FT \mu(k)=0\,\tn{$k$-a.e.}$] $\Rightarrow$ [$\mu=0$] and thus by linearity of $\FT$ this implies that knowing $(\FT \nu)$ $k$-almost everywhere uniquely determines $\nu$.
		Let $g$ be a continuous convolution kernel with compact support and total mass $1$, and for $\veps>0$ let $g_\veps(x) \assign \veps^{-d} \cdot g(x/\veps)$ be the re-scaled version. Then clearly $g_\veps * \mu \weakstar \mu$ as $\veps \to 0$. Let $\varphi \in \c_c(\R^d)$ be continuous with compact support.
		Then one finds
		\begin{align*}
			\int_{\R^d} \varphi\,\diff \mu & = \lim_{\veps \searrow 0} \int_{\R^d} \varphi(x)\,(g_\veps * \mu)(x)\,\diff x
			= \lim_{\veps \searrow 0} (2\pi)^d \int_{\R^d} \overline{(\FT \varphi)(k)}\, (\FT (g_\veps * \mu))(k)\,\diff k \\
			& = \lim_{\veps \searrow 0} (2\pi)^{2d} \int_{\R^d} \overline{(\FT \varphi)(k)}\, (\FT g_\veps)(k) (\FT \mu)(k)\,\diff k
			= 0,
		\end{align*}
		where we first used unitarity (up to normalization) of the Fourier transform on $L^2(\R^d,\C)$ and that $\varphi, (g_\veps * \mu) \in L^2(\R^d,\C)$ for all $\veps>0$, then again the convolution theorem as above, and finally the assumption $\FT \mu(k)=0$ for almost all $k$. Since this holds for all $\varphi \in \c_c(\R^d)$, we must have that $\mu=0$.
		
		In conclusion, the minimizer $\nu$ constructed above must be unique and therefore no discrete minimizers can exist.
	\end{proof}
	Note that the above argument with the convolution only works since $\spt \rho \supset B(0,\kappa \pi/2) + \spt \nu$ (where the plus denotes the Minkowski sum). In other cases, the convolution $\Cos^2_\kappa * \nu$ may not be fully known and consequently $\nu$ may be non-unique.
	
	\subsection[Asymptotic behaviour for kappa to 0]{Asymptotic behaviour for $\kappa \to 0$}
	\label{sec:Asymptotic}
	Now we look more closely at the limiting behaviour of the functional as $\kappa \to 0$ (the case $\kappa \to \infty$ is given by the Wasserstein limit and well-understood).
	We start by specifying the unique minimizer $\nu_\kappa$ for $\kappa = 0$.
	
	\begin{proposition}\label{prop:HellBarycenter}
		Let $\rho \in \prob(\Omega)$ and consider the decomposition
		\begin{equation}\label{eq:rhoDecomposition}
			\rho = \rho_c + \sum_{i=1}^\infty m_i \delta_{x_i}, \quad x_i \in \Omega, m_i \geq 0 \tn{ for all } i \geq 1,
		\end{equation}
		with $\rho_c \in \measp(\Omega)$ atomless. Then,
		\[
		\nu = \sum_{i=1}^\infty m_i^2 \delta_{x_i} \quad \tn{ is the unique optimzier of } \quad \inf \left\{ \int_{\Omega} \Hell^2(\delta_x,\nu)\,\diff\rho(x) \middle| \nu \in \measp(\Omega) \right\}.
		\] 
	\end{proposition}
	
	\begin{proof}
		Taking into account \eqref{eq:rhoDecomposition}, for any $\nu \in \measp(\Omega)$, we can write
		\[
		\int_{\Omega} \Hell^2(\delta_x,\nu)\,\diff\rho(x) = \sum_{i=1}^\infty m_i \Hell^2(\delta_{x_i}, \nu) + \int_{\Omega} \Hell^2(\delta_x,\nu)\,\diff\rho_c(x).
		\]
		For the first term, with \eqref{eq:Hell} observe that for any $x \in \Omega$,
		\begin{align*}
			\Hell^2(\delta_x,\nu) & = (1-\sqrt{\nu(\{x\})})^2+\nu(\Omega \setminus \{x\})
			= 1-2\sqrt{\nu(\{x\})}+\|\nu\|.
		\end{align*}
		For the second term, let $m = \sum_i m_i \in [0,1]$ be the total mass of the atomic part of $\rho$.
		Observe now that, since $\rho_c$ is atomless, we have $\nu(\{x\}) = 0$ for $\rho_c$-a.e. $x \in \Omega$, so that the second term in the sum above simplifies into
		\[
		\int_{\Omega} \Hell^2(\delta_x,\nu)\,\diff\rho_c(x) = \|\rho_c\|(1+\|\nu\|) = \left(1-m\right)(1+\|\nu\|).
		\]
		Therefore, one obtains
		\begin{align*}
			\int_{\Omega} \Hell^2(\delta_x,\nu)\,\diff\rho(x) = 1+\|\nu\|-2\sum_{i=1}^\infty m_i \sqrt{\nu(\{x_i\})}.
		\end{align*}
		Hence, any optimal $\nu$ must be supported on $\{x_i\}_{i=1}^\infty$, and the Hellinger barycenter problem for $\rho$ reduces to
		\[
		\inf \left\{ 1+\sum_{i=1}^\infty n_i -2\sum_{i=1}^\infty m_i \sqrt{n_i} \middle| \nu = \sum_{j=1}^\infty n_j \delta_{x_j}, n_j \geq 0 \right\}.
		\]
		The result follows by first order optimality conditions for each $n_i$.
	\end{proof}
	
	For some $\rho \in \prob(\Omega)$ and $\kappa \in [0,\infty)$ let now $\nu_\kappa$ be a primal optimizer of \primalr. Then, by Corollary~\ref{cor:zeroconvergence}, as $\kappa \to 0$, $\nu_k$ converges to the unique minimizer $\nu_0$ of the Hellinger barycenter problem for $\rho$ which is specified by Proposition \ref{prop:HellBarycenter}. We find that the only contributions to $\nu_0$ arise from the atoms of $\rho$, all other contributions must tend to $0$ as $\kappa \to 0$. The following Lemma provides a rough estimate on the corresponding rate. It is related to the concentration of $\rho$.
	
	\begin{lemma}
		\label{lemma:massBound}
		Let $\rho \in \prob(\Omega)$. For $\kappa \in (0,\infty)$, let $\nu_\kappa \in \measp(\Omega)$ be a minimizer of \primalr. Denote
		\[
		C_{\rho, \kappa} \assign \sup_{y \in \Omega} \left[ \rho(B(y,\kappa\cdot \pi/2)) \right].
		\]
		Then,
		\begin{equation}\label{eq:massBound}
			\|\nu_k\| \leq 4C_{\rho, \kappa}.
		\end{equation}
	\end{lemma}
	
	\begin{proof}
		Via reverse Jensen's inequality, we have
		\[
		\int_\Omega \sqrt{\int_\Omega \Cos^2(|x-y|/\kappa)\,\diff\nu_\kappa(y)} \,\diff\rho(x) \leq \sqrt{ \int_\Omega  \int_\Omega \Cos^2(|x-y|/\kappa)\,\diff\nu_\kappa(y) \diff\rho(x)} \leq \sqrt{C_{\rho, \kappa} \|\nu_\kappa\|}.
		\]
		Taking into account that the zero measure provides an upper bound for the optimal value, the inequality above provides
		\begin{align*}
			1 = J_{\rho, \kappa}(0) \geq J_{\rho, \kappa}(\nu_\kappa)
			\geq 1 + \|\nu_k\| - 2 \sqrt{C_{\rho, \kappa}} \sqrt{\|\nu_\kappa\|},
		\end{align*}
		so that $\|\nu_k\| \leq 4C_{\rho, \kappa}$, which establishes \eqref{eq:massBound}.
	\end{proof}
	
	If $\rho$ contains atoms, then $\lim_{\kappa \searrow 0} C_{\rho, \kappa}>0$ and $\|\nu_0\|>0$, in agreement with Proposition \ref{prop:HellBarycenter}. If $\rho$ is atomless, then $\lim_{\kappa \searrow 0} C_{\rho, \kappa}=0$ (by outer regularity of Radon measures). The rate will depend on $\rho$, and will be slower, for instance, when $\rho$ is concentrated on a lower-dimensional submanifold. When $\lim_{\kappa \searrow 0} C_{\rho, \kappa}=0$ we also obtain $\psi_\kappa \to 1$ uniformly for dual maximizers via \eqref{eq:pdEquationPsi}.
	
	\begin{remark}[Different limits as $\kappa \to \infty$]\label{rem:KappaZeroLimitExample}
		We briefly resume the discussion of Remark \ref{rem:JointStability}.
		Let $(\kappa_n)_n$ be a positive sequence, converging to $\kappa_\infty=0$, let $x$ be in the interior of $\Omega$, and let $\rho_n$ be convolutions of $\delta_x$ with some compact mollifier, with width going to zero as $n \to \infty$, such that $\rho_n \rightweaks \rho_\infty \assign \delta_x$.
		Then the minimizer of $J_{\rho_n,\kappa_\infty}$ will be $\nu=0$ for all $n<\infty$, whereas it will be $\nu=\delta_x$ for $J_{\rho_\infty,\kappa_n}$ for all $n$ up to $n=\infty$.
	\end{remark}
	
	Finally, by assuming that $\rho$ has an $L^2$-density with respect to the Lebesgue measure, we will now provide a more precise statement on the asymptotic behaviour of $\nu_\kappa$ and $\psi_\kappa$.
	
	\begin{proposition}\label{prop:instability}
		Let $\rho \in \prob(\Omega)$ and assume $\rho \ll \Lebesgue^d\restr \Omega$ with $\diff \rho/\diff \Lebesgue^d \in L^2(\Omega)$. Denote by $C_d \assign \|\Cos^2(|\cdot|)\|_{L^1(\R^d)}$. For any $\kappa \in (0,\infty)$, let $(\nu_\kappa, \psi_\kappa) \in \measp(\Omega) \times \c(\Omega)$ be an optimal pair for \primalr and \dualr. Then, $\nu_\kappa \weakstar 0$ as $\kappa \to 0$ with
		\begin{equation}\label{eq:massBoundL2}
			\|\nu_k\| \leq 4C_d\left\|\tfrac{\diff \rho}{\diff \Lebesgue^d} \right\|_{L^2(\Omega)}^2 \cdot \kappa^d.
		\end{equation}
		In particular, $\nu_\kappa/(C_d\kappa^d) \weakstar \left( \tfrac{\diff \rho}{\diff \Lebesgue^d}\right) ^2 \cdot \Lebesgue^d$ and $\|1-\psi_\kappa\|_\infty =O(\kappa^{d/2})$ as $\kappa \to 0$.
	\end{proposition}
	
	\begin{proof}
		\emph{Step 0. Preliminaries on mollifiers.} For each $\kappa > 0$, consider the continuous function $\eta_\kappa \colon \R^d \to \R$ defined as
		\[
		\eta_\kappa(x) \assign \frac{\Cos^2(|x|/\kappa)}{C_d \kappa^d}.
		\]
		The collection $(\eta_\kappa)_{\kappa > 0}$ provides a family of compactly supported continuous mollifiers, which are also positive and radially symmetric. For any measure $\nu \in \meas(\R^d)$, we define as usual
		\[
		(\eta_\kappa * \nu)(x) \assign \int_{\R^d} \eta_\kappa(x-y)\,\diff \nu(y) \quad \text{for } x \in \R^d
		\]
		and extend such definition to functions $f \in L^p_{\textup{loc}}(\R^d)$, $p \geq 1$, setting $\eta_\kappa * f \assign \eta_\kappa * (f\Lebesgue^d)$. We recall from \cite[Appendix C, Theorem 6]{Evans2010} the following classical results:
		\begin{itemize}
			\item Let $f \in L^p(\R^d)$, $p \geq 1$. Then, $\eta_\kappa * f \in \c(\R^d)$ and $\eta_\kappa * f \to f$ in $L^p(\R^d)$ as $\kappa \to 0$.
			\item Let $f \in \c(\R^d)$. Then, $\eta_\kappa * f \to f$ uniformly on compact subsets of $\R^d$ as $\kappa \to 0$.
		\end{itemize}
		From now on, for simplicity, we denote by $\rho$ also the Lebesgue density of $\rho$, so that $\rho \in L^2(\Omega)$. We extend $\rho$ to $L^2(\R^d)$ by assigning the value $0$ outside of $\Omega$.
		
		\medskip
		\noindent
		\emph{Step 1. Mass bound and dual convergence.} Let $\nu_\kappa \in \measp(\Omega)$ be a minimizer of $J_{\rho, \kappa}$. Lemma \ref{lemma:massBound} combined with the absolute continuity of $\rho$ with respect to $\Lebesgue^d$ provides $\|\nu_k\| \leq \Lebesgue^d(B(0,\pi/2)) \|\rho\|_{L^2(\Omega)}^2 \cdot \kappa^{d/2}$, which shows a decay rate which is slower than the stated rate. Hence, the estimate of Lemma \ref{lemma:massBound} has to be refined. By minimality of $\nu_\kappa$ we have $J_{\rho, \kappa}(\nu_\kappa) \leq J_{\rho, \kappa}(0) = 1$ and so, by the Cauchy--Schwartz inequality in $L^2(\Omega)$, we obtain
		\begin{align*}
			1 &= J_{\rho, \kappa}(0) \geq J_{\rho, \kappa}(\nu_\kappa) = 1 + \|\nu_k\| - 2 \langle \sqrt{C_d\kappa^d (\eta_\kappa * \nu_\kappa)}, \rho \rangle_{L^2(\Omega)} \\
			& \geq 1 + \|\nu_k\| - 2 \left\|\sqrt{C_d\kappa^d (\eta_\kappa * \nu_\kappa)}\right\|_{L^2(\Omega)} \|\rho \|_{L^2(\Omega)}
			\geq 1 + \|\nu_k\| - 2 \sqrt{C_d\kappa^d} \sqrt{\|\nu_\kappa\|} \|\rho \|_{L^2(\Omega)},
		\end{align*}
		where we used $\|\eta_\kappa * \nu_\kappa\|_{L^1(\Omega)}=\|\nu_\kappa\|$, so that $\|\nu_k\| \leq 4C_d\|\rho \|_{L^2(\Omega)}^2 \cdot \kappa^d$, which establishes \eqref{eq:massBoundL2}. In particular, thanks to \eqref{eq:pdEquation}, we also have
		\[
		(1-\psi_\kappa(x))^2 = \int_\Omega \Cos^2(|x-y|/\kappa)\,\diff\nu_\kappa(y) \leq \|\nu_\kappa\|
		\]
		and by recalling that $\psi_\kappa \leq 1$, see \eqref{eq:dualr}, and with \eqref{eq:massBoundL2} this establishes $\|1-\psi_\kappa\|_\infty=O(\kappa^{d/2})$ as $\kappa \to 0$.
		
		\medskip
		\noindent
		\emph{Step 2. Energy bound.} Consider $\hat\nu_\kappa = C_d\kappa^d\rho^2\Lebesgue^d \in \measp(\Omega)$. One finds
		\begin{equation}\label{eq:energyCandidate}
			J_{\rho, \kappa}(\hat\nu_\kappa) = 1 + C_d\kappa^d \left( \|\rho\|_{L^2(\Omega)}^2 - 2 \langle \sqrt{\eta_\kappa * \rho^2}, \rho \rangle_{L^2(\Omega)} \right).
		\end{equation}
		By strong convergence of the mollified functions, we have $\eta_\kappa * \rho^2 \to \rho^2$ in $L^1(\R^d)$ as $\kappa \to 0$.
		Using that $(a-b)^2 \leq |a^2-b^2|$ for any $a,b \geq 0$, we eventually obtain that
		\[
		\sqrt{\eta_\kappa * \rho^2} \to \rho \text{ in } L^2(\R^d) \text{ as } \kappa \to 0, \quad \text{which implies}\quad \langle \sqrt{\eta_\kappa * \rho^2}, \rho \rangle_{L^2(\Omega)} = \|\rho\|_{L^2(\Omega)}^2 + o(1).
		\]
		Thus, substituting this expansion into \eqref{eq:energyCandidate}, one gets
		\begin{equation}\label{eq:energyBound}
			\min_{\nu \in \measp(\Omega)} J_{\rho, \kappa}(\nu) \leq J_{\rho, \kappa}(\hat\nu_\kappa) = 1 - C_d\kappa^d \|\rho\|_{L^2(\Omega)}^2 + o(\kappa^d).
		\end{equation}
		
		\medskip
		\noindent
		\emph{Step 3. Convergence of the rescaled minimizers.}
		Let $\nu_\kappa \in \measp(\Omega)$ be a minimizer of $J_{\rho, \kappa}$. Using \eqref{eq:energyBound} we estimate
		\begin{align*}
			0 &= \frac{J_{\rho, \kappa}(\nu_\kappa) - \min_\nu J_{\rho, \kappa}(\nu)}{C_d\kappa^d}
			\geq \frac{\|\nu_\kappa\| - 2\langle \sqrt{C_d\kappa^d (\eta_\kappa * \nu_\kappa)}, \rho \rangle_{L^2(\Omega)} + C_d\kappa^d \|\rho\|_{L^2(\Omega)}^2 + o(\kappa^d)}{C_d\kappa^d} \\
			&= \frac{ \left\| \sqrt{\eta_\kappa * \nu_\kappa} \right\|_{L^2(\R^d)}^2 }{C_d\kappa^d} - 2\left\langle \sqrt{\frac{\eta_\kappa * \nu_\kappa}{C_d\kappa^d}}, \rho \right\rangle_{L^2(\R^d)} + \|\rho\|_{L^2(\R^d)}^2 + o(1) \\
			&= \left\| \sqrt{\frac{\eta_\kappa * \nu_\kappa}{C_d\kappa^d}} - \rho \right\|_{L^2(\R^d)}^2 + o(1),
		\end{align*}
		where we used again $\|\nu_\kappa\|=\|\eta_\kappa * \nu_\kappa\|_{L^1(\Omega)}$ from the first to the second line.
		Therefore, $\sqrt{(\eta_\kappa * \nu_\kappa)/(C_d\kappa^d)}\allowbreak \to \rho$ in $L^2(\R^d)$ as $\kappa \to 0$. In particular,
		\[
		\frac{\eta_\kappa * \nu_k}{C_d\kappa^d} \to \rho^2 \text{ in } L^1(\R^d) \text{ as } \kappa \to 0.
		\]
		Fix any $\phi \in \c(\Omega)$ and consider a continuous bounded extension $\tilde\phi \in \c(\R^d)$ such that $\tilde\phi \restr \Omega = \phi$. We compute
		\begin{align}
			\int_\Omega \phi \, \diff \left( \frac{\nu_k}{C_d\kappa^d} \right) &= \int_{\R^d} \tilde\phi \, \diff \left( \frac{\eta_\kappa * \nu_k}{C_d\kappa^d}\Lebesgue^d \right) + \int_{\R^d} \tilde\phi \, \diff \left( \frac{\nu_k}{C_d\kappa^d} - \frac{\eta_\kappa * \nu_k}{C_d\kappa^d}\Lebesgue^d \right) \nonumber\\
			&= \int_{\R^d} \tilde\phi(x) \cdot \left( \frac{(\eta_\kappa * \nu_k)(x)}{C_d\kappa^d} \right)\,\diff x + \int_{\R^d} (\tilde\phi(x) - (\eta_\kappa * \tilde\phi)(x)) \, \diff \left(\frac{\nu_k(x)}{C_d\kappa^d} \right)
			\label{eq:expansion} \\
			&\to \int_{\R^d} \tilde\phi(x) \cdot \rho^2(x) \,\diff x = \int_\Omega \phi(x) \cdot \rho^2(x) \,\diff x \quad \text{ as } \kappa \to 0,\nonumber
		\end{align}
		where the second term in \eqref{eq:expansion} converges to $0$ because $\tilde\phi - (\eta_\kappa * \tilde\phi)$ converges uniformly to $0$ on $\Omega$ and $\|\nu_k/(C_d\kappa^d)\| \leq 4\|\rho \|_{L^2(\Omega)}^2$ by means of \eqref{eq:massBoundL2}. Hence, $\nu_\kappa/(C_d\kappa^d) \weakstar \rho^2\Lebesgue^d$ in $\measp(\Omega)$ as $\kappa \to 0$.
	\end{proof}

	If $\rho$ is substituted by a weak* approximation $\hat{\rho}$, Corollary \ref{cor:minimizersconvergence} implies that minimizers $\hat{\nu}$ for $\hat{\rho}$ converge (up to subsequences) to a minimizer $\nu$ for $\rho$ as $\hat{\rho} \rightweaks \rho$.
	When $\rho$ is atomless, for small $\kappa$, by Proposition \ref{prop:HellBarycenter} and Lemma \ref{lemma:massBound} we conclude that $\nu$ (and thus also $\hat{\nu}$) is close to the zero measure.
	If we are now interested in the `residuals' of $\nu$ and $\hat{\nu}$ (i.e., if we re-scale them such that their mass is on the order 1), then Proposition \ref{prop:instability} tells us that we can only expect the residual of $\hat{\nu}$ to be close to that of $\nu$ when $\hat{\rho}$ approximates $\rho$ well in an $L^2$-sense.
	Therefore, if we were interested in using the $\HK$-barycenter to obtain a quantization or clustering of some measure $\rho$ at a small $\kappa$-scale, but only an approximation $\hat{\rho}$ is available, then the approximate solution $\hat{\nu}$ will only be useful, if $\hat{\rho}$ is a good approximation in an $L^2$-sense.
	(Intuitively, we expect that it is sufficient if the $L^2$-approximation holds after an optional convolution with a mollifier at a scale less than $\kappa$.)
	
	\section{Numerical examples}
	\label{sec:Numerics}
	To obtain a better understanding of the behaviour of the HK barycenter between Dirac measures and to illustrate the theoretical results of the previous section we now consider some numerical approximations.
	
	\subsection{Lagrangian optimization scheme}
	\label{sec:NumericsScheme}
	Let us first consider the discrete barycenter problem between $r \in \N$ unit Dirac measures on $\Omega$ with $\mu_i = \delta_{x_i}$, $x_i \in \Omega$, and weights $\lambda_i>0$ for $i=1,\ldots,r$  such that $\sum_{i=1}^r \lambda_i=1$. This corresponds to $\Lambda\assign \sum_{i=1}^r \lambda_i\,\delta_{\mu_i}$ in \eqref{eq:primal} or equivalently $\rho \assign \sum_{i=1}^r \lambda_i\,\delta_{x_i}$ in \eqref{eq:primalr}.
	
	For optimization over $\nu$ we employ a Lagrangian discretization, i.e.~we optimize over the ansatz $\nu^s = \sum_{j=1}^{s} m_j \delta_{y_j}$ with locations $y_j \in \Omega$ and masses $m_j\geq 0$ for $j=1,\ldots,s$ for some $s \in \N$. The number of points in the ansatz~$s$ may change during optimization due to merging or addition of new particles.
	The resulting optimization problem can be written as
	\begin{equation}
		\label{eq:ProblemDiscrete}
		\min_{m_j \ge 0,\;y_j} J_{\rho, \kappa}(\nu^s)  = \min_{m_j \ge 0,\;y_j} 1 + \sum_{j=1}^s m_j - 2 \sum_{i = 1}^{r} \lambda_i \sqrt{\sum_{j=1}^{s} m_j \Cos^2\left(\frac{|x_i - y_j|}{\kappa}\right)}.
	\end{equation}
	
	For gradient-based minimization we determine the partial derivatives with respect to mass coefficients~$m_j$ and locations~$y_j$. The components of the gradient in mass are
	\begin{equation*}
		\frac{\partial J_{\rho, \kappa}(\nu^s)}{\partial m_j} = 1 - \sum_{i = 1}^{r} \lambda_i
		\frac{\Cos^2\left(\frac{|x_i - y_j|}{\kappa}\right)}{\sqrt{\displaystyle \sum_{l=1}^{s} m_l \Cos^2\left(\frac{|x_i - y_l|}{\kappa}\right) }}.
	\end{equation*}
	If we set $\psi^s(x_i) \assign 1- \sqrt{\sum_{j=1}^s m_j \Cos^2\left(\frac{|x_i-y_j|}{\kappa}\right)  }$ (which would be the optimal dual $\psi$ if $\nu^s$ is primal optimal, see Proposition \ref{prop:DualityR}), then we obtain
	\begin{equation*}
		\frac{\partial J_{\rho, \kappa}(\nu^s)}{\partial m_j} = 1 - F_{\rho,\kappa}(\psi^s)(y_j),
	\end{equation*}
	i.e.~masses need to be increased when the corresponding constraint function at $y_j$ is less than 1 (the constraint is inactive) or decreased when it is violated. A vanishing gradient corresponds to the optimality condition $F_{\rho,\kappa}(\psi)(y_j)=1$ on the support of $\nu$.
	
	For the gradient in coordinates $y_j$ one obtains
	\begin{equation*}
		\frac{\partial J_{\rho, \kappa}(\nu^s)}{\partial y_j} = 
		\sum_{i=1}^{r} \lambda_i
		\frac{2 m_j \, \Cos\left(\frac{|x_i-y_j|}{\kappa}\right)
			\Sin\left(\frac{|x_i-y_j|}{\kappa}\right) \frac{y_j-x_i}{\kappa |x_i-y_j|}}{ \sqrt{\sum_{l=1}^{s} m_l \displaystyle \Cos^2\left(\frac{|x_i - y_l|}{\kappa}\right)}},
	\end{equation*}
	where $\Sin(x)\assign \sin(x)$ for $x\in [0,\pi/2]$ and $0$ otherwise.
	By comparison we find again a relation to the constraint function (with the same $\psi$ as above),
	\begin{equation*}
		\frac{\partial J_{\rho, \kappa}(\nu^s)}{\partial y_j} =  - \frac{\partial F_{\rho,\kappa}(\psi^s)(y_j)}{\partial y_j},
	\end{equation*}
	i.e.~the points $y_j$ will move `upwards' on the constraint function $F_{\rho,\kappa}(\psi)$ and only be locally optimal when sitting at a maximum, which then, by the mass optimality condition has to be at value 1.
	
	Due to the Lagrangian ansatz, the resulting optimization problem is non-convex and may get stuck in non-optimal points. On the other hand, because the points are allowed to move the spatial accuracy is not limited to a grid. The issue of poor local minima can be remedied by testing whether the value of $F_{\rho,\kappa}(\psi)$ exceeds one at points where no $y_j$ is located. This testing can be performed numerically with reasonable accuracy since $F_{\rho,\kappa}(\psi)$ is $1/\kappa$-Lipschitz continuous. Thus it is possible to combine the strengths of Lagrangian and Eulerian schemes:
	We remove points from the ansatz when their mass drops to zero, and we may add points when the dual constraint is violated, thus adaptively determining the appropriate number of point masses.
	
	A Eulerian discretization with entropic smoothing was used in \cite{friesecke2019barycenters} for numerical examples. In Proposition \ref{prop:DiscreteBar} and \cite[Proposition 6.2]{friesecke2019barycenters} it was shown that discrete minimizers exist when $\rho$ is discrete. The entropic Eulerian ansatz cannot approximate them with high accuracy due to the fixed grid, entropic blur, and since we observe that minima are often `shallow' or non-unique. Therefore, to study these discrete minimizers the non-entropic Lagrangian method is more appropriate.
	
	The problem formulated above is then solved with (preconditioned) gradient descent with gradient steps performed simultaneously in coordinates and in masses; inexact line search as described in~\cite{HagerZhang2005}.
	
	Naturally, we are also interested in examples where $\rho$ is not discrete, representing an uncountable infinite number of input measures. While Corollary \ref{cor:minimizersconvergence} suggests that this case can be approximated with discrete $\rho$, we deduce from Proposition \ref{prop:instability} that for small $\kappa$ it will be difficult to obtain good approximations for the `residual' of $\nu_\kappa$ (which will be close to the zero measure).
	Therefore, for this regime we employ a slightly different numerical scheme where $\nu$ is also approximated in a (discrete) Lagrangian fashion but the integral over $\rho$ in \eqref{eq:primalrDirect} is approximated more accurately by adaptive Gauss--Kronrod quadrature instead of individual Dirac sums as in \eqref{eq:ProblemDiscrete}. The corresponding formulas for the gradients are derived in analogy.
	For optimization of this functional we applied the quasi-Newton BFGS algorithm.

	\subsection{Finite number of input measures}
	For a discrete number of input Dirac measures, for $\kappa$ sufficiently close to zero, it is easy to see that the resulting HK barycenter will be a superposition of Dirac measuers, one per input measure (cf.~Prop.~\ref{prop:HellBarycenter}). For $\kappa$ sufficiently large, it will be one single Dirac measure.
	For $\kappa$ increasing from 0 to $\infty$, a gradual merging of Diracs in the barycenter was observed in \cite{friesecke2019barycenters} on various example. Here we demonstrate numerically that the general behaviour is more complex. As $\kappa$ increases, Diracs may merge and split, disappear and reappear, and the total number of Diracs may even temporarily increase.
	
	\begin{figure}[hbt]
		\centering
		\includegraphics[width=0.58\textwidth]{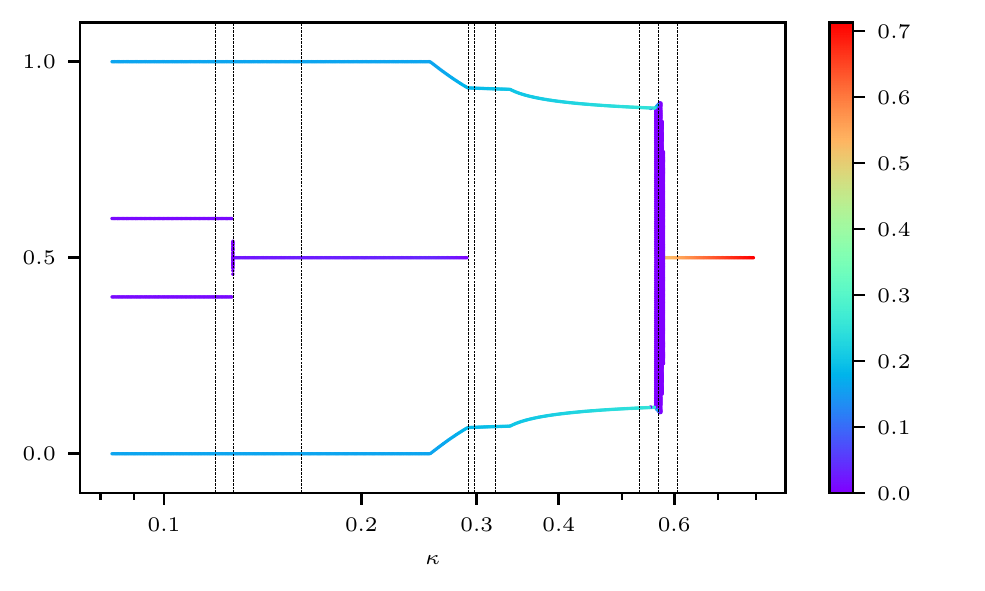}
		\includegraphics[width=0.36\textwidth]{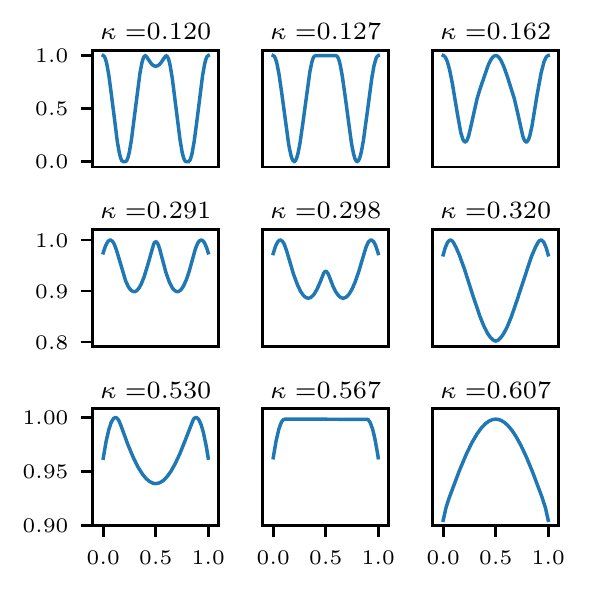}
		\caption{Left: the HK barycenter on $\Omega=[0,1]$ for $\rho$ and $\kappa$ as in \eqref{eq:RhoExample1}. For each $\kappa$, position of points indicates positions of Dirac measures, color code indicates the amount of mass. The vertical lines show locations for which the constraint function $F_{\rho,\kappa}(\psi)$ is shown on the right.}
		\label{fig:4masses}
	\end{figure}
	\FloatBarrier
	
	Figure \ref{fig:4masses} illustrates the HK barycenter for
	\begin{align}
		\label{eq:RhoExample1}
		\rho & \assign 0.4 \cdot \delta_0 + 0.1 \cdot \delta_{0.4} + 0.1 \cdot \delta_{0.6} + 0.4 \cdot \delta_1 \tn{ on } \Omega=[0,1], &
		\tn{for } \kappa & \in [0.08,0.8],
	\end{align}
	and the constraint function $F_{\rho,\kappa}(\psi)$ for the corresponding dual optimal $\psi$.
	For small $\kappa$, as expected, the barycenter consists of four individual Dirac masses at the same locations as in $\rho$. Eventually the two middle masses merge (where the constraint function briefly exhibits an extended plateau of value 1, as analyzed in \cite{friesecke2019barycenters}). At some point, the outer masses `see' the inner masses (i.e.~their relative distance drops below $\kappa \pi/2$). Since their $\lambda$-weights are much higher, the joint Dirac in the barycenter remains much closer to the outer masses until the Dirac at the center even vanishes. Note that after this vanishing, the constraint function briefly even exhibits a local maximum at $0.5$ which is strictly below 1. Eventually, all masses merge into one cluster. During the merging the constraint function exhibits an extended plateau of value 1 for an extended interval of $\kappa$ values and in this regime a non-discrete barycenter exists (again, this was already analyzed in \cite{friesecke2019barycenters}).
	
	\begin{figure}[hbt]
		
		\begin{subfigure}{0.61\textwidth}	
			\centering	
			\includegraphics[width=\textwidth]{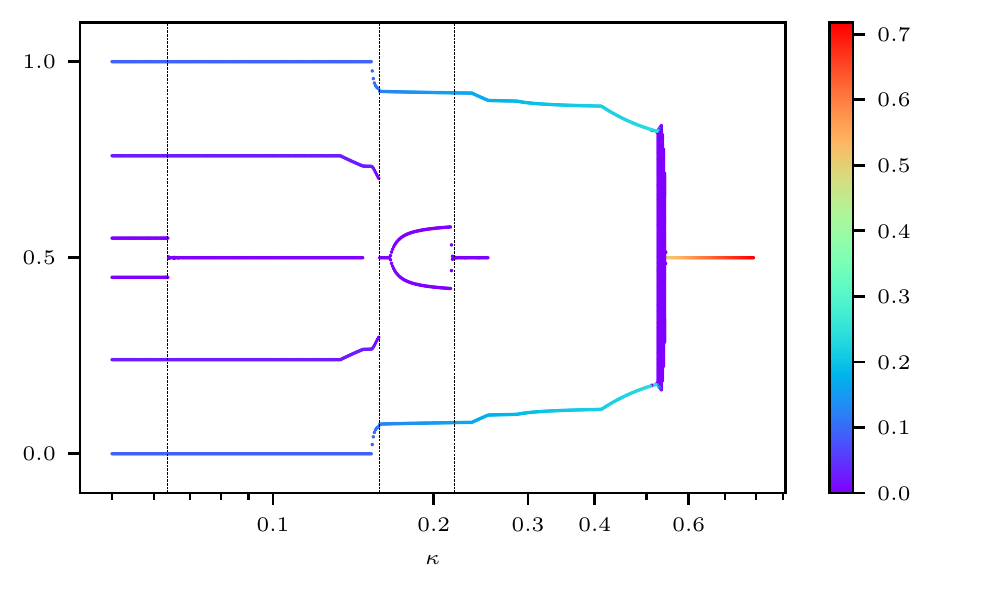}	
		\end{subfigure}	
		\begin{subfigure}{0.36\textwidth}	
			\centering	
			\includegraphics[width=\textwidth]{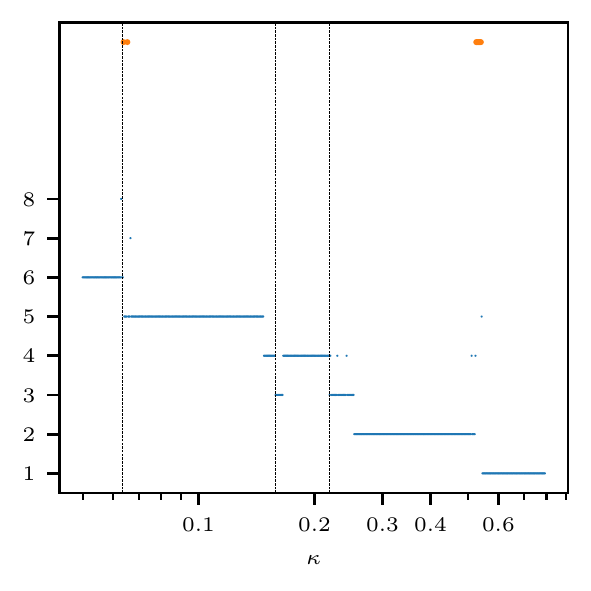}
		\end{subfigure}
		\caption{Left: the HK barycenter on $\Omega=[0,1]$ for a $\rho$ with six Dirac masses (see text), visualized as in Figure \ref{fig:4masses}. Right: the number of masses in the barycenter $\nu$, which is not decreasing over $\kappa$. The regimes with a seemingly diffuse solution are marked by orange points.}
		\label{fig:6masses}
	\end{figure}
	\FloatBarrier
	
	Figure \ref{fig:6masses} shows a similar example with 6 masses in $\rho$, given by
	$$\rho \assign 0.3 \cdot (\delta_0+\delta_1) + 0.16 \cdot (\delta_{0.24}+\delta_{0.76}) + 0.03 \cdot (\delta_{0.45}+\delta_{0.55}).$$
	The trajectory of HK barycenters over $\kappa$ exhibits an even more intricate behaviour with the mass at the center appearing and disappearing several times and even the number of masses temporarily increasing as $\kappa$ increases. For at least two regions of scales, a diffuse barycenter seems to be admissible.
	
	\subsection{A continuum of input measures}
	\label{sec:NumericsContinuous}
	
	Now we consider a continuum of input measures. Let $\Omega = [0,1]$ and $\rho = \Lebesgue^1\restr\Omega$. Following Section \ref{sec:NumericsScheme} we consider the functional
	\begin{equation}
		J_{\rho, \kappa}^s(y_1, \dots, y_s, m_1, \dots, m_s) = J_{\rho, \kappa}(\nu^s) = 1 + \sum_{j=1}^s m_j -2 \int_{\Omega} \sqrt{\sum_{j=1}^s m_j \Cos^2(|x-y_j|/\kappa)}\,\diff x,
	\end{equation}
	which corresponds to \eqref{eq:ProblemDiscrete} with continuous $\rho$. The integral is approximated by adaptive Gauss--Kronrod quadrature and minimized via projected BFGS in positions and masses.
	
	Figure \ref{fig:uniform} shows barycenters obtained for $\kappa \in [1/12, 1]$ and the evolution of the total mass. In agreement with Proposition \ref{prop:instability} the latter decreases linearly to $0$ as $\kappa \to 0$.
	The evolution seems to consist of intervals in which the number of Dirac masses gradually decreases by one step at a time until only a single mass is left. However, these intervals are separated by transition regions, during which the behaviour is more complicated and the number of masses also temporarily increases again.
	
	\begin{figure}[tbh]
		\begin{subfigure}{0.49\textwidth}
			\centering
			\includegraphics[width=\textwidth]{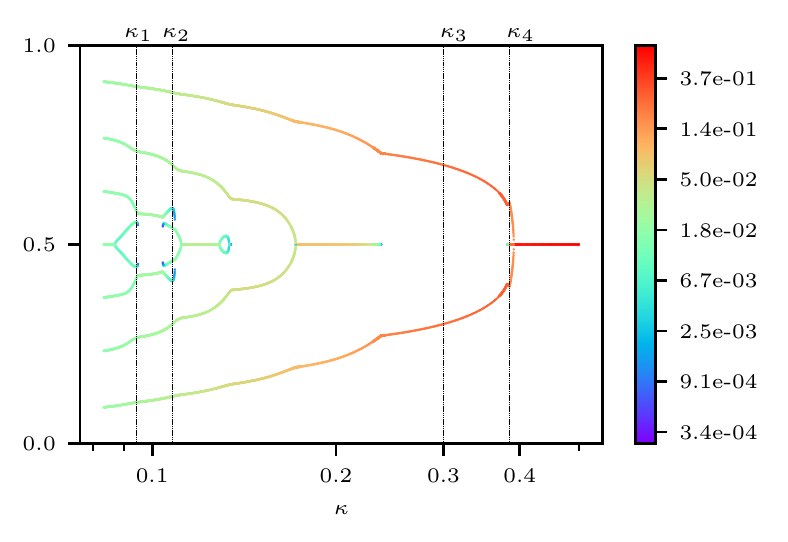}
		\end{subfigure}
		\begin{subfigure}{0.49\textwidth}
			\centering
			\includegraphics[width=\textwidth]{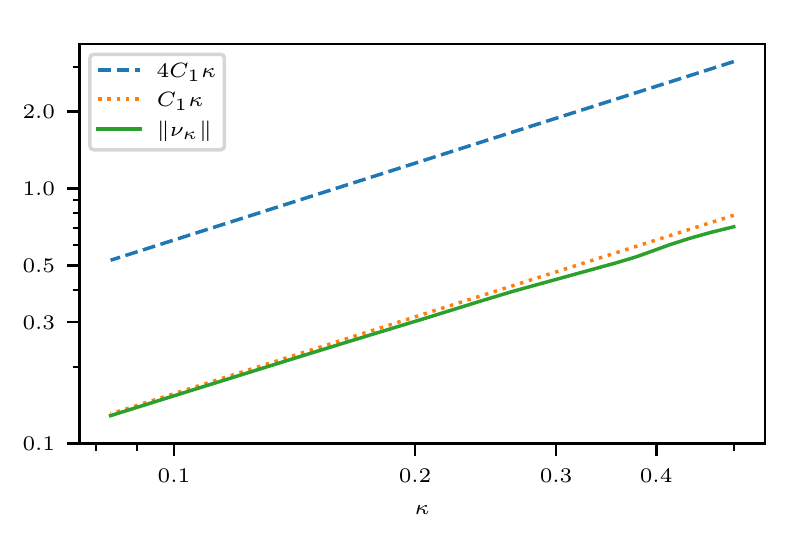}
		\end{subfigure}
		\\
		\begin{subfigure}{\textwidth}
			\centering
			\includegraphics[width=\textwidth]{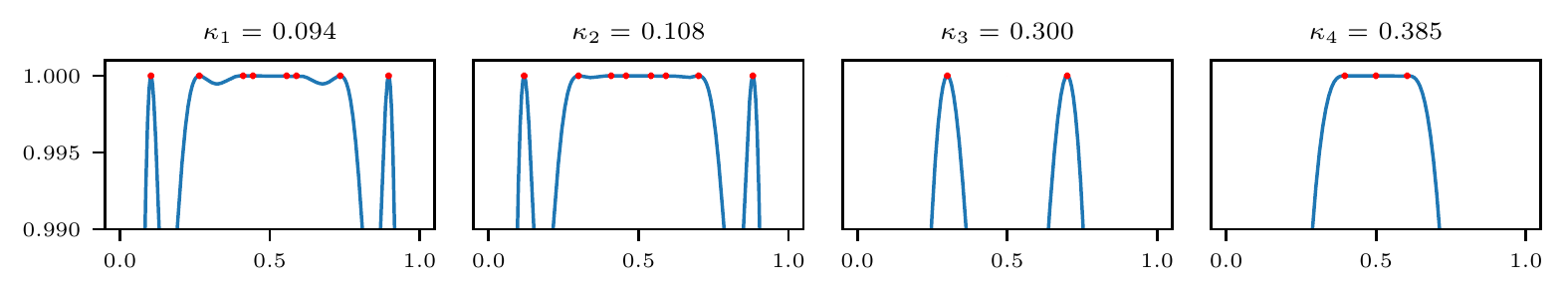}
		\end{subfigure}
		\caption{Top-Left: the HK barycenter on $\Omega=[0,1]$ for a $\rho=\Lebesgue \restr [0,1]$ visualized as in Figure \ref{fig:4masses} (here with a logarithmic colour map). Top-Right: total mass of the HK barycenter, in comparison with bound and asymptotic expansion from Proposition \ref{prop:instability}. Bottom: the constraint function $F_{\rho,\kappa}(\psi)$ for some values of $\kappa$ (as marked in the top-left), with positions $(y_j)_j$ of the primal masses marked by red points (note the range of the vertical axis, which only shows a very small interval close to one).}
		\label{fig:uniform}
	\end{figure}
	\FloatBarrier

	Due to the uniformity of $\rho$ this problem proved to be quite challenging numerically, as the constraint function for the optimal $\psi$ was very close to one, almost throughout the entire bulk of the interval, see Figure \ref{fig:uniform}. In particular the transition regions required detailed manual inspection.
	It is possible to solve the problem analytically for very small and very large $\kappa$, but the full spectrum seems to be beyond reach. Therefore, it seems ultimately impossible to prove that the true minimizers have the same structure as our numerical approximations. But via the primal-dual optimality conditions we can at least guarantee that the numerical approximations must be very close in terms of objective value. In particular, the observed complicated transitions seem to outperform simpler variants without additional particles. These transitions are shown in more detail in Figure \ref{fig:uniform_zoom}.

	\begin{figure}[tbh]
		\begin{subfigure}{0.32\textwidth}
			\centering
			\includegraphics[width=\textwidth]{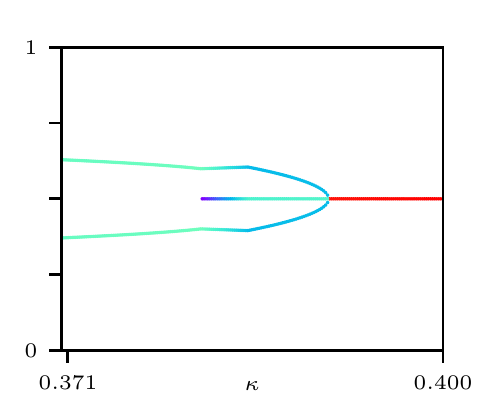}
		\end{subfigure}
		\begin{subfigure}{0.32\textwidth}
			\centering
			\includegraphics[width=\textwidth]{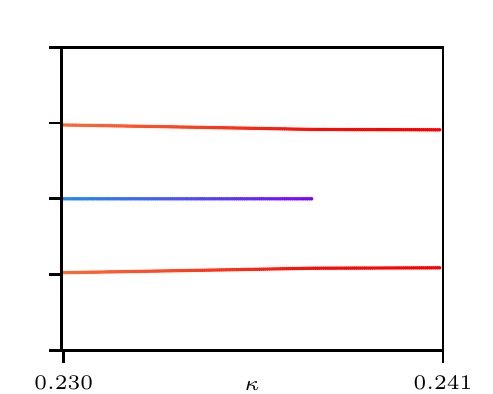}
		\end{subfigure}
		\begin{subfigure}{0.32\textwidth}
			\centering
			\includegraphics[width=\textwidth]{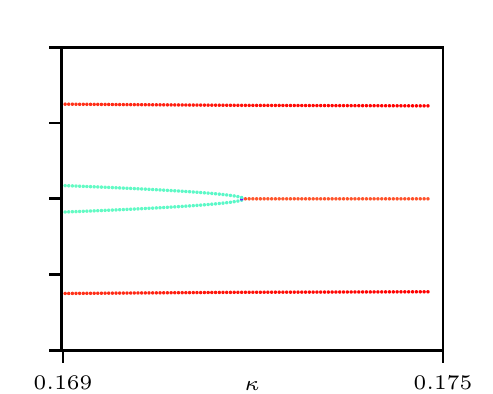}
		\end{subfigure}
		\\
		\begin{subfigure}{0.32\textwidth}
			\centering
			\includegraphics[width=\textwidth]{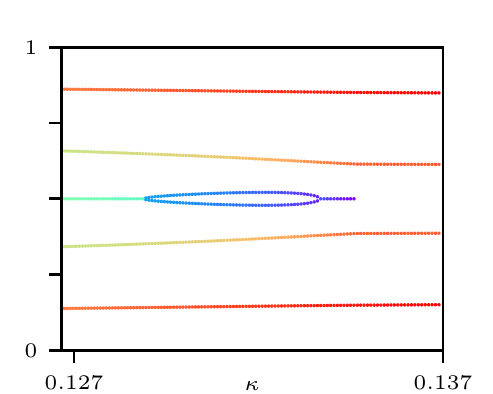}
		\end{subfigure}
		\begin{subfigure}{0.32\textwidth}
			\centering
			\includegraphics[width=\textwidth]{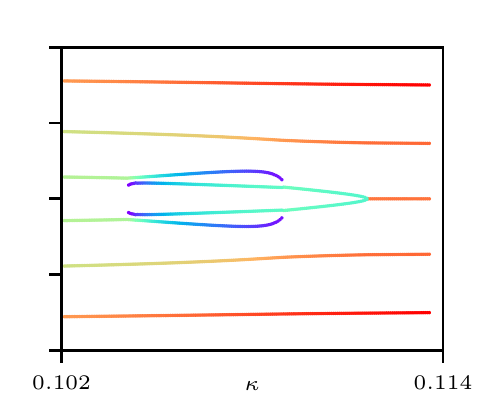}
		\end{subfigure}
		\begin{subfigure}{0.32\textwidth}
			\centering
			\includegraphics[width=\textwidth]{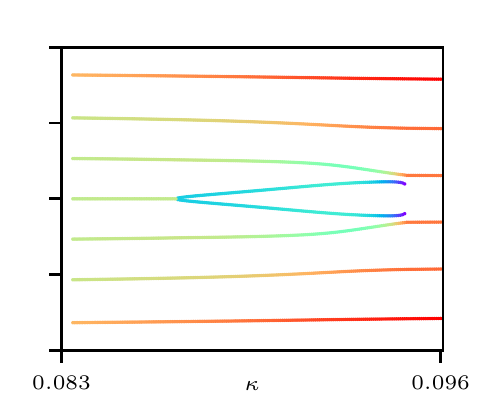}
		\end{subfigure}
		\caption{Zoom on the transition regions for Figure \ref{fig:uniform} for the given values of $\kappa$. For better visibility the color scale is adjusted to the mass range in each sub-figure.}
		\label{fig:uniform_zoom}
	\end{figure}
	\FloatBarrier

	It seems that each of the shown transitions follows a different pattern: From one to two particles, first a fork into three particles is observed, and then the middle particle vanishes (numerically it seems that in this region also a diffuse solution would be admissible, but we were unable to find a solution with less than three particles). In the transition from two to three, the third particle simply appears in the middle. From three to four, the middle particle splits into two. From four to five, a new particle first appears, then splits, and finally re-merges. From five to six, a particle first splits, but the two fragments then vanish and are replaced by appearing new particles. From six to seven, two particles appear and eventually merge. We did not anticipate such a complicated structure in a convex functional.
	
	\subsection{Comparison with empirical measures}
	
	Next, we study the convergence of the $\HK$ barycenter as $\rho$ is approximated through sampling. For the previous example with $\rho$ being the uniform measure on $[0,1]$ we now generate $\hat{\rho}$ by drawing $n$ points from $\rho$ and using the obtained empirical measure. Corresponding empirical barycenters are shown in Figure~\ref{fig:uniform_sampled}. 
	
	\begin{figure}[tbh]
		\begin{subfigure}{0.49\textwidth}
			\centering
			\includegraphics[width=\textwidth]{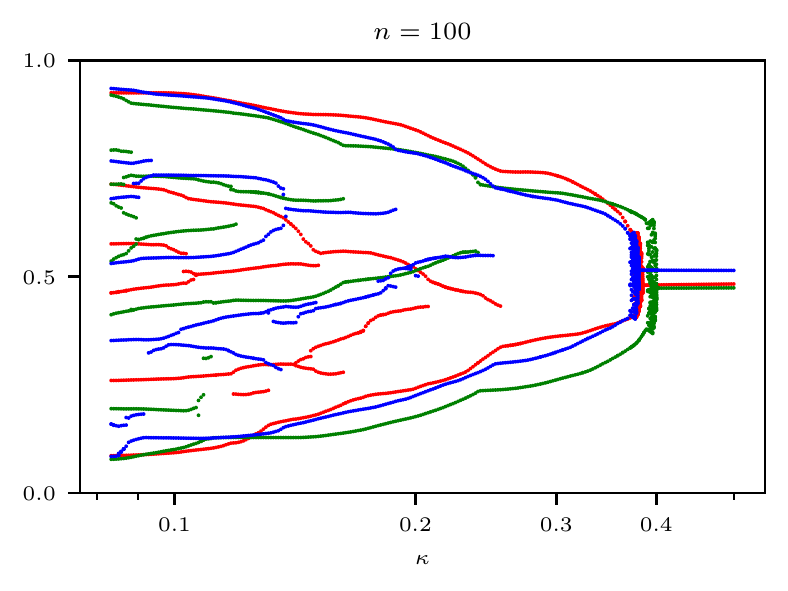}
		\end{subfigure}
		\begin{subfigure}{0.49\textwidth}
			\centering
			\includegraphics[width=\textwidth]{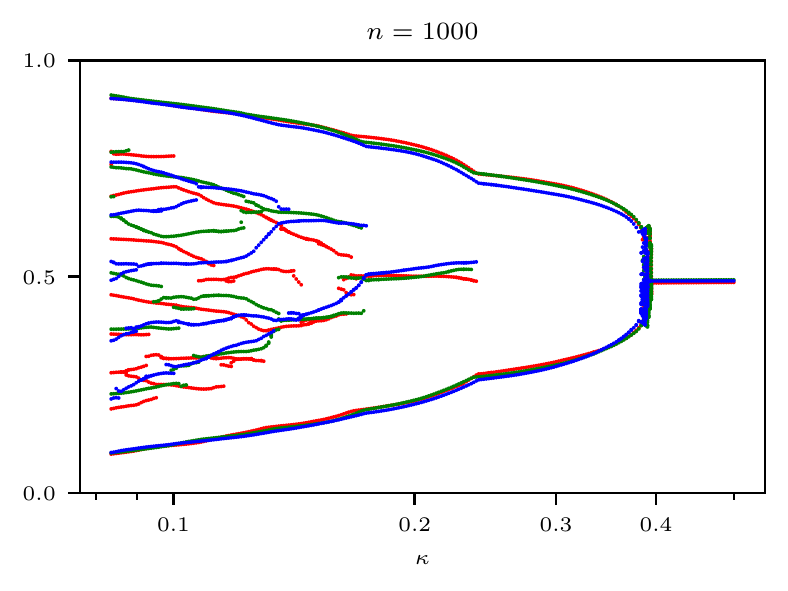}
		\end{subfigure}
		\caption{Barycenters for the input measures sampled from uniform distribution on $[0,1]$. Points mark support of masses, mass itself is not visualized. Three different instances are shown in different colors to visualize the variance between them. Left: 100 points sampled. Right: 1000 points sampled. }
		\label{fig:uniform_sampled}
	\end{figure}
	\FloatBarrier
	
	By Corollary \ref{cor:minimizersconvergence} we expect convergence of the empirical barycenter to the true one, as $n \to \infty$. However, as $\kappa \to 0$, the true barycenter will converge to the zero measure (Proposition \ref{prop:HellBarycenter}). Convergence of the `residual' part is analyzed in Proposition \ref{prop:instability} for the case when $\rho$ has an $L^2$-density. From this we expect that the residual of the empirical barycenter will be close to the real residual, when $\hat{\rho}$ is a good $L^2$-approximation of $\rho$. As $\hat{\rho}$ is an empirical measure, it has no $L^2$-density. Intuitively, we expect the result to still hold when a small convolution on a length scale below $\kappa$ is applied to $\hat{\rho}$, and when this mollified version of $\hat{\rho}$ is close to $\rho$ in an $L^2$-sense. In a nutshell, we expect the residuals to become worse, as $\kappa$ decreases, and better as $n$ increases. This is confirmed by the examples in Figure \ref{fig:uniform_sampled}.
	Convergence of the dual solution and constraint function (Prop.~\ref{prop:DualStability}) is visualized in Figure \ref{fig:psi-ffunction}.
	
	\begin{figure}[tbh]
		\begin{subfigure}{0.49\textwidth}
			\centering
			\includegraphics[width=\textwidth]{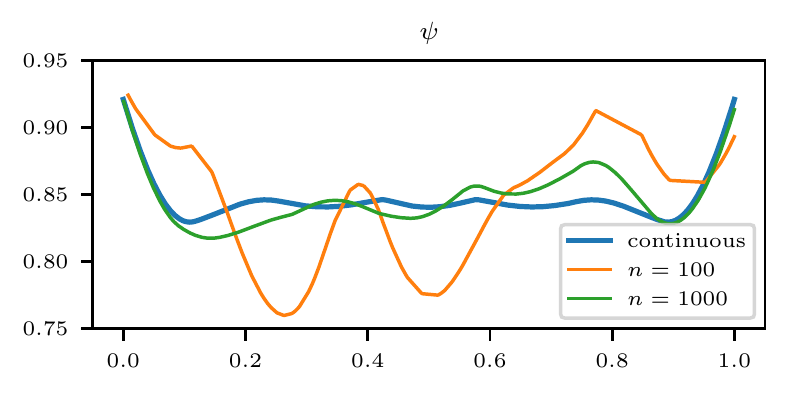}
		\end{subfigure}
		\begin{subfigure}{0.49\textwidth}
			\centering
			\includegraphics[width=\textwidth]{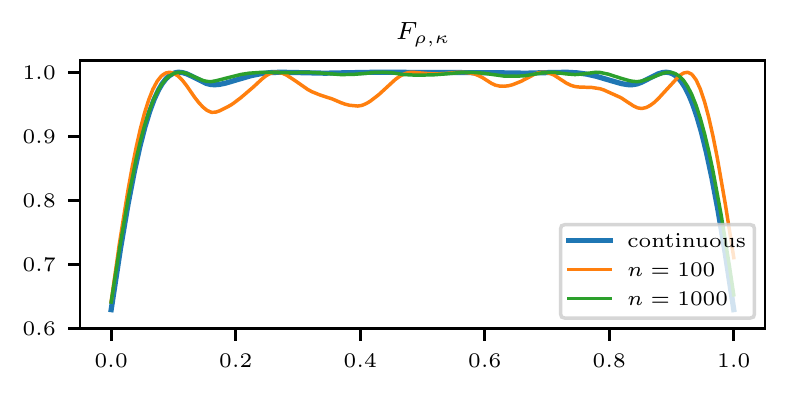}
		\end{subfigure}
		\caption{The dual (left) and the constraint function (right) at $\kappa=0.1$, for two sampled solutions with $n=100$ and $n=1000$ together with the solution obtained from the continuous ansatz of Section \ref{sec:NumericsContinuous} for comparison. }
		\label{fig:psi-ffunction}
	\end{figure}
	\FloatBarrier
	
	The uniformity of $\rho$ in the example above makes it is not quite clear what kind of `clustering' to expect for smaller $\kappa$. Therefore, we perform similar experiments on mixtures of Gaussian distributions. That is, $\rho$ is given by a mixture of 5~Gaussians with mean and standard deviations given by
	$$(0.15, 0.05),\; (0.30, 0.03),\;  (0.46, 0.08),\;  (0.71, 0.03), \; (0.81, 0.06).$$
	
	Figure \ref{fig:Gaussians} shows the corresponding numerical results for sampling $n_i=100$ and $n_i=1000$ points from each Gaussian.
	The coarse structure of the resulting HK barycenters seems to indicate five major clusters, one per Gaussian, that gradually merge. As expected, cluster masses are higher for more concentrated Gaussians, and the second cluster seems to absorb some points from the first Gaussian. As in the earlier examples, the transition between the major cluster intervals is more complicated, and differs between instances. Occasionally, smaller spurious particles with low mass are present. The rough structure of the barycenters seems consistent between all four experiments, in agreement with the proven stability results.

	\begin{figure}[hbt]
		\centering
		\includegraphics[width=\textwidth]{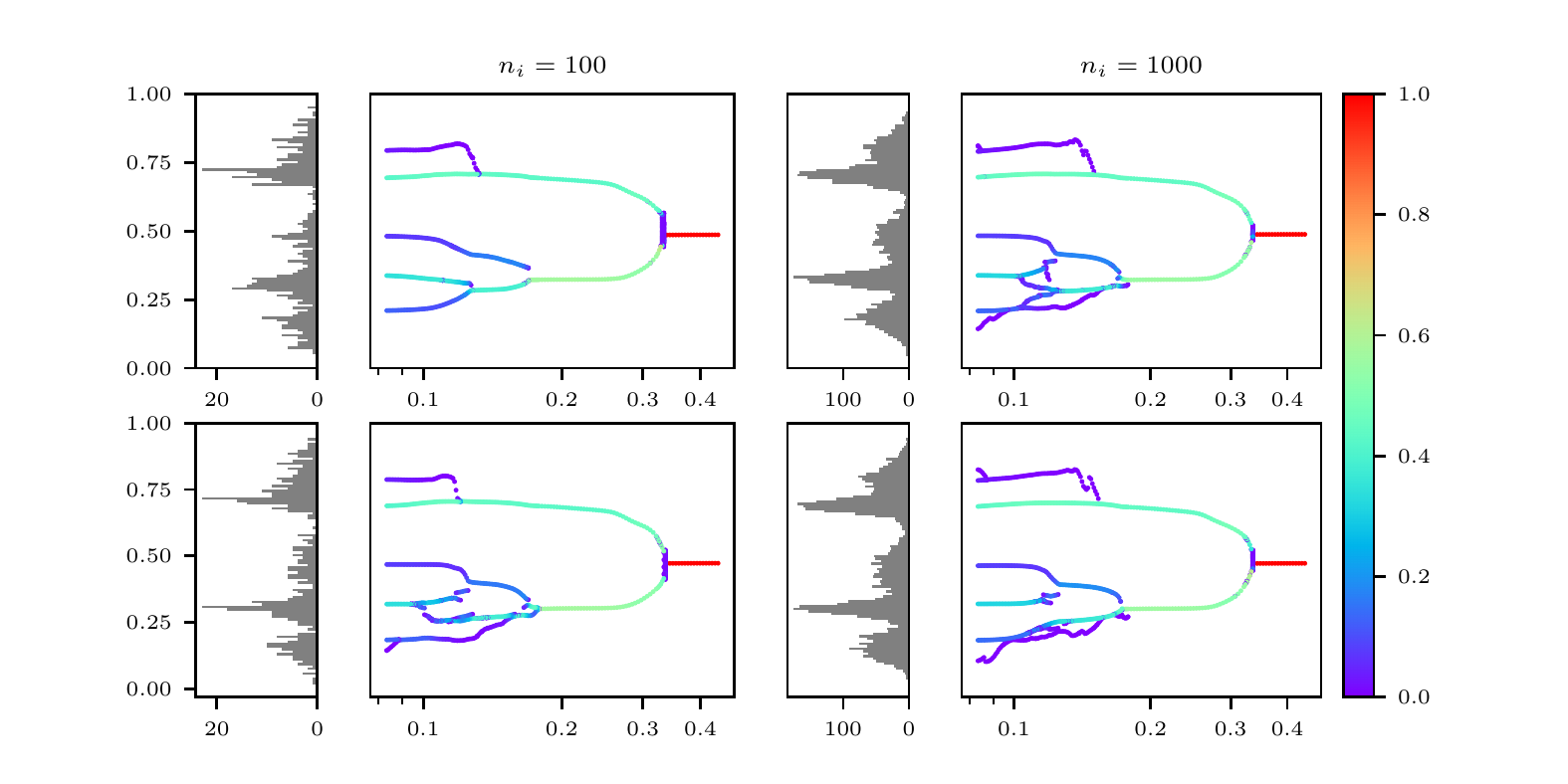}
		\caption{HK barycenters for samples from a mixture of 5 Gaussian distributions. Each column shows two instances. Left column: 100 samples per Gaussian. Right column: 1000 samples per Gaussian. Empirical distribution is visualized by the vertical histograms for each instance (50 bins in left column, 100 bins in right column). The mass of the barycenter for each $\kappa$ is re-normalized to 1 for better visibility. }
		\label{fig:Gaussians}
	\end{figure}
	\FloatBarrier
	
	Figure~\ref{fig:Gaussians_F_function} shows the regions of the constraint function $F_{\rho,\kappa}(\psi)$ close to one for the examples presented in Figure~\ref{fig:Gaussians}. The threshold is scaled with parameter~$\kappa$ as this seems to produce lines of approximately constant width over the scales. By Proposition \ref{prop:DualStability} we know that these regions are unique and convergent as $n \to \infty$. They seem to be in good correspondence with the primal pictures and are reasonably stable under repeated experiments.
	
	\begin{figure}[hbt]
		\centering
		\includegraphics[width=0.7\textwidth]{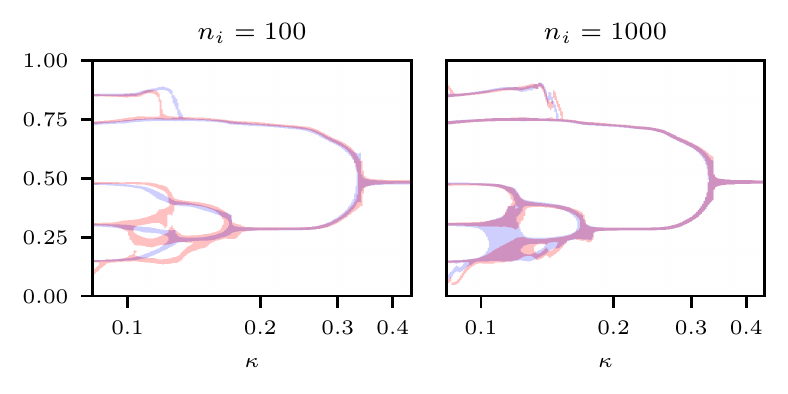}
		\caption{Thresholded regions where the dual constraints are close to being active, i.e.~where $F_{\rho,\kappa}(\psi) \geq 1-\frac{\exp(-9.5)}{\kappa}$
			for two instances with $n_i=100$ (left) and $n_i=1000$ (right) sampled points in each Gaussian presented in~Figure~\ref{fig:Gaussians}. 
			The re-scaling with $\kappa$ was done based on the empirical observation that it yielded approximately consistent widths of the lines.}
		\label{fig:Gaussians_F_function}
	\end{figure}
	\FloatBarrier
	
	For visual comparison, Figure~\ref{fig:clusters} shows four single linkage cluster dendrograms for the sampled points presented above, computed with the algorithm described in~\cite{olson1995parallel}. The obtained major clusters are qualitatively similar to the results obtained by the HK barycenter. It differs from the HK barycenter figure in some important features: First, it is well known that single linkage produces many spurious outliers that show up in the figure as dark, thin lines. The HK barycenter seems less prone to such outliers. Second, the obtained dendrograms for $n_i=100$ and $n_i=1000$ appear to be shifted horizontally against each other, since the expected distances between pairs of points are different in both cases. For the HK barycenter the behaviour is qualitatively the consistent at a fixed $\kappa$ for different $n_i$. Third, of course the dendrogram does provide a strict hierarchical clustering of the data, unlike the HK barycenters, and it can be computed with simple and efficient algorithms. These are features that the HK barycenter does not offer.
	For an approach to making the dendrograms more robust to spurious outliers, see for instance~\cite{ClusterTreeEstimation2014}.
	
	\begin{figure}[tbh]
		\centering
		\includegraphics[width=0.7\textwidth]{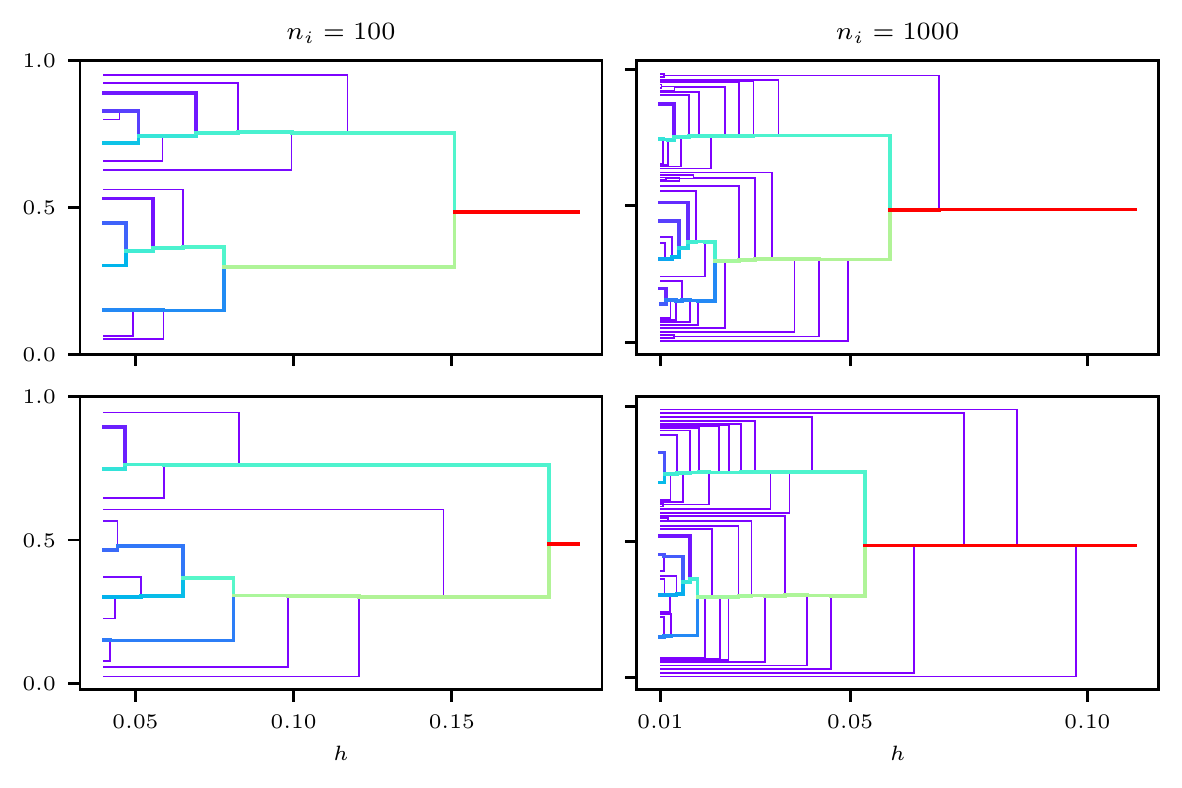}
		\caption{Single linkage cluster dendrograms for two instances with $n_i=100$ (left) and for 2 instances with $n_i=1000$ (right) sampled points in each Gaussian presented in~Figure~\ref{fig:Gaussians}, truncated for better visibility. Colorscale represents the number of points clustered in a branch. Branches carrying less then $2\%$ of the total mass are shown with narrower lines for better visibility.}
		\label{fig:clusters}
	\end{figure}
	\FloatBarrier

	\subsection{A two-dimensional example}
	
	Finally, we present some two-dimensional examples. We start with the uniform density on the square~$[0,1]^2$, discretized by $131^2$ discrete Dirac masses, see Figure~\ref{fig:2Dsquare}. As in one dimension, an intricate sequence of transitions between relatively regular intervals is observed. 
	
	\begin{figure}[tbh]
		\centering
		\includegraphics[width=\textwidth]{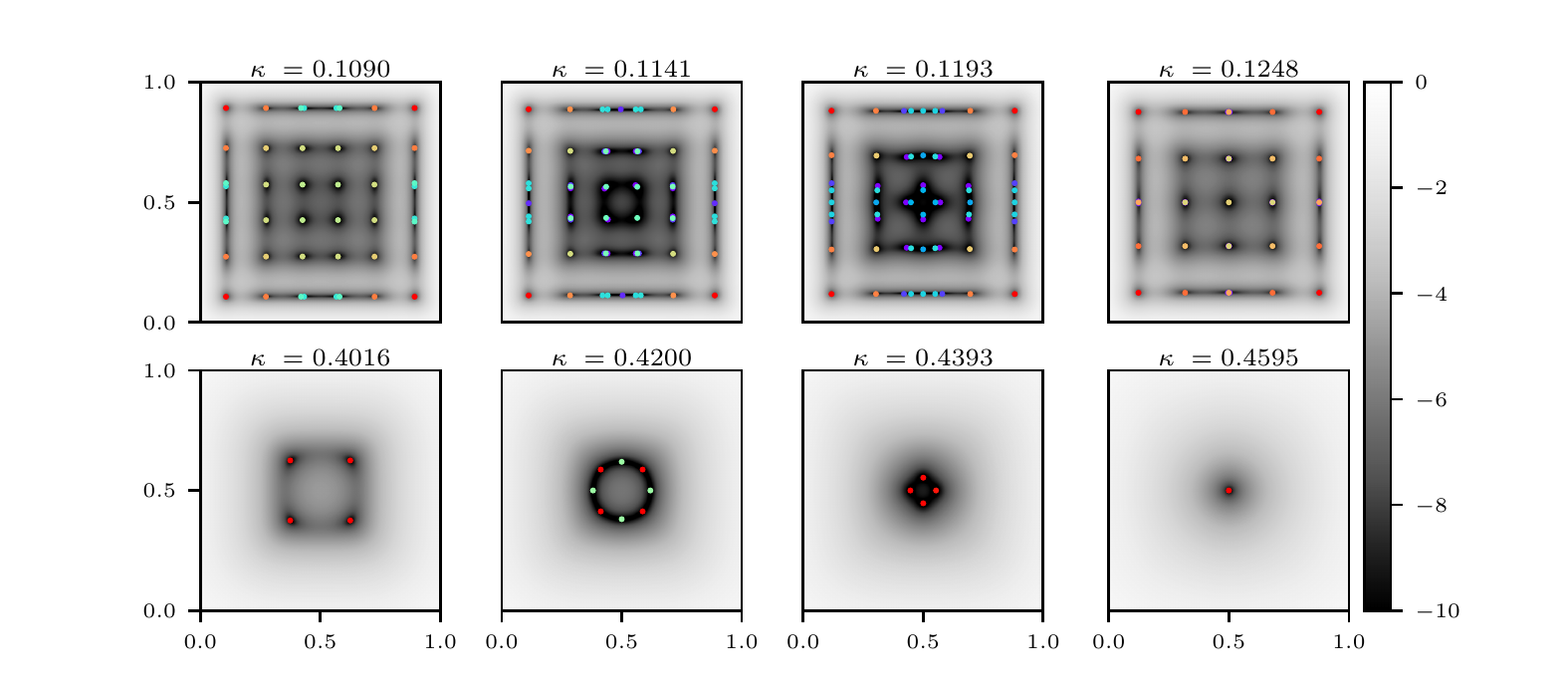}
		\caption{HK barycenter for various $\kappa$ for uniform $\rho$. In grayscale: dual feasibility residual in log-scale ($\ln(|1-F_{\rho,\kappa}(\psi)|)$). In color scale: locations and masses of barycenter with maximal mass in each barycenter re-normalized to 1 for better visibility.}
		\label{fig:2Dsquare}
	\end{figure}
	\FloatBarrier

	An example with a mixture of 3 Gaussians is shown in Figure~\ref{fig:2DGaussians}. For this experiment, 50 points were sampled from each 2D Guassian distribution. The sampled points are shown in the plot in grey, the barycenter for selected $\kappa$ is shown in color. The HK barycenter in this case presents a 'clustering' behavior similar to the one shown for a mixture of Gaussians in~1D.
	\begin{figure}[tbh]
		\centering
		\includegraphics[width=\textwidth]{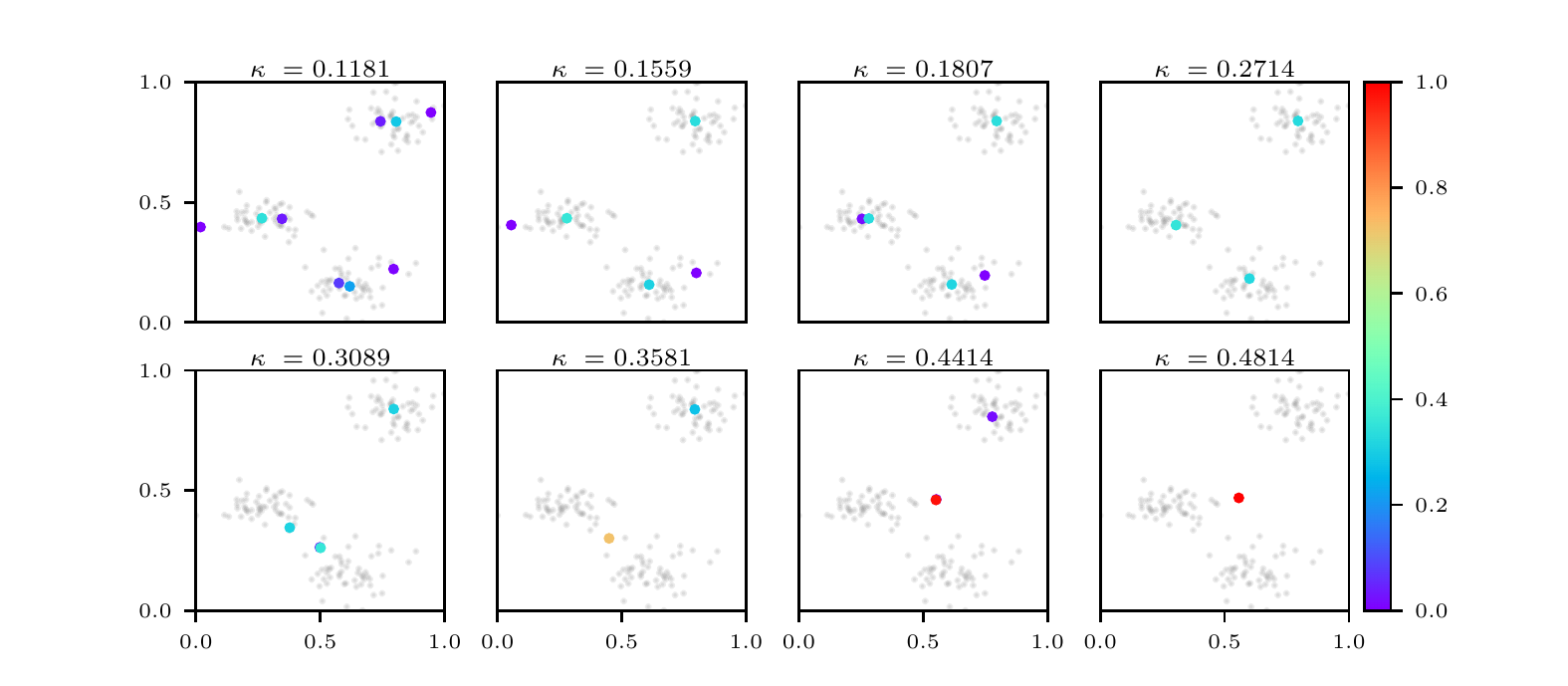}
		\caption{HK barycenter for various $\kappa$ for $\rho$ sampled from mixture of Gaussians in 2D. Sampled points are shown in grey. The mass of the barycenter for each~$\kappa$ is re-normalized to 1 for better visibility. }
		\label{fig:2DGaussians}
	\end{figure}
	\FloatBarrier

	\section{Conclusion}
	In this article we have studied in more detail the barycenter between an uncountable number of input measures with respect to the Hellinger--Kantorovich distance, with a particular focus on Dirac input measures.
	We have shown existence and stability with respect to input data and length scale parameter and derived a corresponding dual problem.
	For Dirac input measures we have shown existence of continuous dual maximizers, their uniqueness ($\rho$-a.e.) and primal-dual optimality conditions. The behaviour of the solutions as $\kappa \to 0$ was studied in more detail, including the limit solution and asymptotic mass and density estimates. We showed that in some cases no discrete minimizers can exist. A numerical scheme based on Lagrangian discretization was introduced and it was shown numerically that the evolution of the minimizer with respect to the length scale does not correspond to a simple gradual merging of `clusters'. With these two properties (non-existence of discrete minimizers, no simple merging behaviour) the HK barycenter does not induce a simple hierarchical clustering of data points in the conventional sense. Instead, a wide variety of transition behaviors is observed numerically. However, it still provides a one-parameter family of measures, interpolating between the input data and a single Dirac measure, which can be interpreted as a gradual coarse graining. It is reasonably robust under empirical approximation by sampling, as demonstrated theoretically and with numerical examples, and comes with a corresponding family of dual problems that provide additional interpretation and information.
	
	As such it might be an interesting tool for the structure analysis of point clouds and application to real data would be a possible direction for future research. This would lead to related questions such as the interpretation of the trajectory of barycenters in high dimensions and their reliable numerical approximation.
	
	\paragraph{Acknowledgement.}
	This research was supported by the DFG through the Emmy Noether Programme (project number 403056140) and the project `Nonsmooth and nonconvex optimal transport problems' within the Priority Programme SPP 1962: `Complementarity-Based Distributed Parameter Systems: Simulation and Hierarchical Optimization' (project number 423447095).
	\label{sec:Conclusion}
	
	\bibliography{references}{}
	\bibliographystyle{plain}
	
\end{document}